\documentclass[11pt,reqno,a4paper]{amsart}
\usepackage{amssymb,amscd,amsmath}
\usepackage{pspicture}
\usepackage{graphicx}
\usepackage{mathptmx}

\usepackage[matrix,arrow,curve,frame]{xy}    

\xymatrixcolsep{1.9pc}                          
\xymatrixrowsep{1.9pc}
\newdir{ >}{{}*!/-5pt/\dir{>}}                  

 \calclayout

\makeatletter

\renewcommand{\subsection}{\@startsection{subsection}{1}{0pt}{-3.25ex plus -1ex minus-.2ex}{1.5ex plus.2ex}{\normalfont\it}}
\renewcommand{\section}{\@startsection{section}{1}{\parindent}{3.5ex plus 1ex minus .2ex}{2.3ex plus.2ex}{\sc}}

\renewcommand{\phi}{\varphi}
\renewcommand{\leq}{\leqslant}
\renewcommand{\geq}{\geqslant}
\renewcommand{\epsilon}{\varepsilon}

\renewcommand{\kappa}{\varkappa}
 
\DeclareMathOperator{\spec}{Spec}
 \DeclareMathOperator{\cyl}{cyl}
 \DeclareMathOperator{\Cyl}{Cyl}

 \DeclareMathOperator{\Map}{Map}

\DeclareMathOperator{\Hom}{Hom} 
 
 \DeclareMathOperator{\id}{id}

\DeclareMathOperator{\Ho}{Ho}

 \DeclareMathOperator{\kr}{Ker}
 \DeclareMathOperator{\im}{Im}
\DeclareMathOperator{\coker}{Coker} \DeclareMathOperator{\nis}{nis}
 \DeclareMathOperator{\Ar}{Ar}
 
 \DeclareMathOperator{\Mod}{Mod}
 \DeclareMathOperator{\Ob}{Ob}
\DeclareMathOperator{\supp}{Supp}

\newcommand{\lra}[1]{\bl{#1}\longrightarrow\relax}
\newcommand{\bl}[1]{\buildrel #1\over}
\newcommand{\cc}{\mathcal}
\newcommand{\bb}{\mathbb}
\newcommand{\ps}{\oplus}
\newcommand{\ff}{\mathfrak}

\newcommand{\op}{{\textrm{\rm op}}}

\newcommand{\wh}{\widehat}
\newcommand{\wt}{\widetilde}

\newtheorem{thm}{Theorem}[section]
\newtheorem{prop}[thm]{Proposition}
\newtheorem*{sublem}{Sublemma}
\newtheorem{cor}[thm]{Corollary}
\newtheorem{lem}[thm]{Lemma}

\newtheorem{rem}[thm]{Remark}

\newtheorem{defs}[thm]{Definition}

\makeatother

\begin{document}

\footskip30pt


\title{On the motivic spectral sequence}
\author{Grigory Garkusha}
\address{Department of Mathematics, Swansea University, Singleton Park, Swansea SA2 8PP, UK}
\email{g.garkusha@swansea.ac.uk}

\urladdr{http://math.swansea.ac.uk/staff/gg/}

\author{Ivan Panin}
\address{St. Petersburg Branch of V. A. Steklov Mathematical Institute,
Fontanka 27, 191023 St. Petersburg, Russia}

\address{St. Petersburg State University, Department of Mathematics and Mechanics, Universitetsky prospekt, 28, 198504,
Peterhof, St. Petersburg, Russia}

\email{paniniv@gmail.com}

\thanks{This paper was written during the visit of the second author to
Swansea University supported by EPSRC grant EP/J013064/1. He would
like to thank the University for the kind hospitality}



\begin{abstract}
It is shown that the Grayson tower for $K$-theory of smooth
algebraic varieties is isomorphic to the slice tower of
$S^1$-spectra. We also extend the Grayson tower to bispectra and
show that the Grayson motivic spectral sequence is isomorphic to the
motivic spectral sequence produced by the Voevodsky slice tower for
the motivic $K$-theory spectrum $KGL$. This solves Suslin's problem
for these two spectral sequences in the affirmative.
\end{abstract}

\maketitle

\thispagestyle{empty} \pagestyle{plain}

\newdir{ >}{{}*!/-6pt/@{>}} 


\section{Introduction}

One of the more significant developments in algebraic $K$-theory in
the 1990-s/early 2000-s was the construction of an algebraic
analogue for the Atiyah--Hirzebruch spectral sequence. It is a
strongly convergent spectral sequence
   $$E_2^{pq}=H_{\cc M}^{p-q,-q}(X,\bb Z)\Longrightarrow K_{-p-q}(X)$$
that relates the motivic cohomology groups of a smooth variety to
its algebraic $K$-groups. The existence of this spectral sequence
was first conjectured by Beilinson~\cite{Bei}. It is also called the
{\it motivic spectral sequence}. Its construction is given in
various forms:

\begin{itemize}
\item[(MSS1)] the Bloch--Lichtenbaum motivic spectral sequence~\cite{BL} for the
spectrum of a field together with the Friedlander--Suslin and Levine
extensions~\cite{FS,Levloc} to the global case for a smooth variety
over a field;

\item[(MSS2)] the Grayson motivic spectral sequence~\cite{Gr,Sus,Wlk,GP};

\item[(MSS3)] the Voevodsky motivic spectral sequence~\cite[p.~171]{DLORV}
produced by the slice filtration of the motivic $K$-theory spectrum
$KGL$~\cite{VoeProbl,VoeAppr}.
\end{itemize}

A problem of Suslin says that the three types of the motivic
spectral sequences agree with each other. In~\cite{Levcon} Levine
solved the Voevodsky problem for the slices of the spectrum
$KGL$~\cite{VoeProbl,VoeAppr} (over a perfect field). As a
consequence he shows that (MSS1) agrees with (MSS3) over perfect
fields.

In this paper we show that over perfect fields the Grayson tower for
$K$-theory of smooth algebraic varieties agrees with the slice tower
of $S^1$-spectra (see Theorem~\ref{grs1tower}). The Grayson tower is
then extended to bispectra. Thanks to this it is proved that (MSS2)
agrees with (MSS3) (over perfect fields), answering the Suslin
problem in the affirmative for these two spectral sequences (see
Theorem~\ref{ochenhorosho}).

To conclude the introduction, we make the following remark
recommended by the referee. In~\cite{Pod}, Podkopaev claims that
(MSS1) agrees with (MSS2) by comparing Fried\-lander--Suslin's and
Grayson's towers. He shows in six steps that the entries of both
towers agree but does not show the agreement of the towers maps, on
which the differentials in both spectral sequences depend. It may be
possible to compare the maps in the future.

Throughout the paper we denote by $Sm/k$ the category of smooth
separated schemes of finite type over the base field $k$.

\section{Preliminaries}

In this section we collect basic facts about the $K$-theory
associated with cubes of additive categories. We mostly follow
Grayson~\cite{Gr}.

\subsection{Bivariant additive categories}
\label{Bivariant}

Let $AddCats$ denote the category of small additive categories and
additive functors. Let $AffSm/k$ be the full subcategory of $Sm/k$
whose objects are the affine smooth $k$-schemes. By a {\it bivariant
additive category\/} we mean a functor
   $$\cc A:(Sm/k)^{\op}\times AffSm/k\to AddCats.$$
So to any $X\in Sm/k$ and $Y\in AffSm/k$ we associate an additive
category $\cc A(X,Y)$ which is contravariant in $X$ and covariant in
$Y$.

We also require that there is an action of $AffSm/k$ on $\cc A$
in the following sense. Given $U\in AffSm/k$ there is an additive
functor
   $$\alpha_U:\cc A(X,Y)\to\cc A(X\times U,Y\times U),$$
functorial in $X$ and $Y$, such that for any morphism $f:U\to V$ in
$AffSm/k$ the following square of additive functors is strictly
commutative
   $$\xymatrix{\cc A(X\times V,Y\times V)\ar[rrr]^{\cc A(1_X\times f,1_{Y\times V})}&&&\cc A(X\times U,Y\times V)\\
               \cc A(X,Y)\ar[u]_{\alpha_V}\ar[rrr]^{\alpha_U}&&&\cc A(X\times U,Y\times U).
               \ar[u]_{\cc A(1_{X\times U},1_Y\times f)}}$$
By the functoriality of $\alpha_U$ in $X$ we mean that the following
square of additive functors is strictly commutative for any $Y\in
AffSm/k$ and any morphism $f: X^{\prime} \to X$ in $Sm/k$
$$\xymatrix{\cc A(X\times U,Y\times U)\ar[rrr]^{\cc A(f\times 1_U,1_{Y\times U})}&&&\cc A(X^{\prime}\times U,Y\times U)\\
               \cc A(X,Y)\ar[u]_{\alpha_U}\ar[rrr]^{\cc A(f,1_Y)}&&&\cc A(X^{\prime},Y).
               \ar[u]_{\alpha_U}}$$
By the functoriality of $\alpha_U$ in $Y$ we mean that the following
square of additive functors is strictly commutative for any $X \in
Sm/k$ and any morphism $g: Y \to Y^{\prime}$ in $AffSm/k$
$$\xymatrix{\cc A(X\times U,Y\times U)\ar[rrr]^{\cc A(1_{X\times U},g\times 1_U)}&&&\cc A(X\times U,Y^{\prime}\times U)\\
               \cc A(X,Y)\ar[u]_{\alpha_U}\ar[rrr]^{\cc A(1_X,g)}&&&\cc A(X,Y^{\prime}).
               \ar[u]_{\alpha_U}}$$

Below we shall associate an explicitly constructed bispectrum to any
bivariant additive category. For this we need to collect some facts
about the algebraic $K$-theory of additive categories.

\subsection{$K$-theory for cubes of additive categories}

We let $Ord$ denote the category of finite nonempty ordered sets.
For $A\in Ord$ we define a category $Sub(A)$ whose objects are the
pairs $(i,j)$ with $i\leq j\in A$, and where there is an (unique)
arrow $(i',j')\to(i,j)$ exactly when $i'\leq i\leq j\leq j'$. Given
an additive category $\cc M$, we say that a functor $M:Sub(A)\to\cc
M$ is {\it additive\/} if $M(i,i)=0$ for all $i\in A$, and for all
$i\leq j\leq k\in A$ the map $M(i,k)\to M(i,j)\oplus M(j,k)$ is an
isomorphism. Here 0 denotes a previously chosen zero object of $\cc
M$. The set of such additive functors is denoted by $Add(Sub(A),\cc
M)$. Given ordered sets $A_1,\ldots,A_n$, we let
$Add(Sub(A_1)\times\cdots\times Sub(A_n),\cc M)$ denote the set of
multi-additive functors, i.e., functors that are additive in each
variable.

The Grayson simplicial set $S^\oplus\cc M$~\cite{Gr2,Gr} is defined
as
   $$(S^\oplus\cc M)(A)=Add(Sub(A),\cc M).$$
An $n$-simplex $M\in S_n^\oplus\cc M$ may be thought of as a
compatible collection of direct sum diagrams $M(i,j)\cong
M(i,i+1)\oplus\cdots\oplus M(j-1,j)$. There is a natural map
$S^\oplus\cc M\to S\cc M$ (see~\cite[p.~147]{Gr}) which converts
each direct sum diagram $M(i,k)\cong M(i,j)\oplus M(j,k)$ into the
short exact sequence $0\to M(i,j)\to M(i,k)\to M(j,k)\to 0$. Here
$S\cc M$ stands for the Waldhausen $S$-construction~\cite{Wal}.

We follow the same constructions as in~\cite[Section~8.7]{Rog} to
define Grayson's symmetric spectrum $K^{Gr}(\cc M)$. Given a
positive integer $n$, one can define the $n$-fold multisimplicial
additive category $S^{\oplus,n}\cc M:=S^{\oplus}\bl n\ldots
S^{\oplus}\cc M$. The $n$th space of Grayson's $K$-theory spectrum
is given by\label{KGr}
   $$K^{Gr}(\cc M)_n=|\Ob(S^{\oplus,n}\cc M)|,$$
where the right hand side is the diagonal of the $n$-fold
multisimplicial set $\Ob (S^{\oplus,n}\cc M)$. The $n$th symmetric
group $\Sigma_n$ acts on $K^{Gr}(\cc M)_n$ by permuting the order of
the $S^{\oplus}$-constructions. Each structure map $\sigma$ is the
composite
   $$|\Ob(S^{\oplus,n}\cc M)|\wedge S^1\cong|\Ob(S^{\oplus,n}S^{\oplus}\cc M)|^{(1)}\subset|\Ob(S^{\oplus,n}S^{\oplus}\cc M)|\cong|\Ob(S^{\oplus,n+1}\cc M)|,$$
where the superscript ${}^{(1)}$ stands for the 1-skeleton with
respect to the last simplicial direction. The $k$-fold iterated
structure map $\sigma^k$ is then defined as the composite
   $$|\Ob(S^{\oplus,n}\cc M)|\wedge S^k\cong|\Ob(S^{\oplus,n}S^{\oplus}\bl k\ldots S^{\oplus}\cc M)|^{(1,\ldots,1)}\subset
     |\Ob(S^{\oplus,n}S^{\oplus}\bl k\ldots S^{\oplus}\cc M)|\cong|\Ob(S^{\oplus,n+k}\cc M)|,$$
where the superscript ${}^{(1,...,1)}$ indicates the
multi-1-skeleton with respect to the $k$ last simplicial directions.
This map is plainly $(\Sigma_n\times\Sigma_k)$-equivariant. With
these definitions $K^{Gr}(\cc M)$ becomes a (semistable) symmetric
spectrum. If $\cc M$ happens to be a multisimplicial additive
category, then we define its Grayson $K$-theory symmetric spectrum
$K^{Gr}(\cc M)$ by taking diagonals $K^{Gr}(\cc
M)_n:=|\Ob(S^{\oplus,n}\cc M)|$ of the multisimplicial sets
$\Ob(S^{\oplus,n}\cc M)$, $n\geq 0$.

In~\cite[\S4]{Gr2} is presented a construction called $C$ which can
be applied to a cube of additive categories to convert it into a
multisimplicial additive category, the $K$-theory of which serves as
the iterated cofiber space/spectrum of the corresponding cube of
$K$-theory spaces/spectra. We start with preparations.

We let $[1]$ denote the ordered set $\{0<1\}$ regarded as a
category, and we use $\epsilon$ as notation for an object of $[1]$.
By an $n$-dimensional cube in a category $\cc C$ we shall mean a
functor from $[1]^n$ to $\cc C$. An object $C$ in $\cc C$ gives a
0-dimensional cube denoted by $[C]$, and an arrow $C\to C'$ in $\cc
C$ gives a 1-dimensional cube denoted by $[C\to C']$. If the
category $\cc C$ has products, we may define an external product of
cubes as follows. Given an $n$-dimensional cube $X$ and an
$n'$-dimensional cube $Y$ in $\cc C$, we let $X\boxtimes Y$ denote
the $n+n'$-dimensional cube defined by $(X\boxtimes
Y)(\epsilon_1,\ldots,\epsilon_{n+n'})=X(\epsilon_1,\ldots,\epsilon_n)\times
Y(\epsilon_{n+1},\ldots,\epsilon_{n+n'})$. Let $\bb G_m^{\wedge n}$
denote the external product of $n$ copies of $[1\to \bb G_m]$ in
$Sm/k$. For example, $\bb G^{\wedge 2}_m$ is the square of schemes
   $$\xymatrix{\spec k\ar[r]\ar[d]&\bb G_m\ar[d]\\
               \bb G_m\ar[r]&\bb G_m\times\bb G_m.}$$

Let $L$ be a symbol, and consider $\{L\}$ to be an ordered set.
Given an ordered set $A$, by $\{L\}A$ we mean the concatenation
ordered set with $L$ declared to be less than every element of $A$.
Given an $n$-dimensional cube of additive categories $\cc M$, we
define an $n$-fold multisimplicial additive category $C^\oplus\cc M$
as a functor from $(Ord^n)^{\op}$ to the category of additive
categories by letting $C^\oplus\cc M(A_1,\ldots,A_n)$ be the
additive category whose objects are the multi-additive natural
transformations (we follow here the terminology of~\cite{Gr})
   \begin{equation}\label{cube}
    Add([Sub(A_1)\to Sub(\{L\}A_1)]\boxtimes\cdots\boxtimes[Sub(A_n)\to Sub(\{L\}A_n)],\cc M).
   \end{equation}
More precisely, every object in~\eqref{cube} maps each vertex of the
cube of the domain to the corresponding vertex of the cube $\cc M$
by means of a multi-additive functor. If we regard each edge of the
cube of the domain as a functor between categories, then one has a
commutative diagram of functors, in which one pair of parallel
arrows is this edge and the corresponding edge (which is an additive
functor) of the cube $\cc M$. When $n=0$, we may identify
$C^\oplus\cc M$ with $\cc M$. We define $S^\oplus\cc M$ to be
$S^\oplus C^\oplus\cc M$, the result of applying the
$S^\oplus$-construction of Grayson degreewise. It is an $n+1$-fold
multisimplicial set (see~\cite{Gr} for details).

If we extend the following lemma to Grayson's $K$-theory spectra in
the obvious way, then we get that Grayson's $K$-theory
$K^{Gr}(C^\oplus\cc M)$ of $C^\oplus\cc M$ serves as the iterated
cofiber space/spectrum of the corresponding cube of Grayson's
$K$-theory spaces/spectra.

\begin{lem}(\cite[4.3]{Gr})\label{grfiltr}
Suppose we are given an additive map $\cc M'\to\cc M$ of
$n$-dimensional cubes of additive categories. Let $[\cc M'\to\cc M]$
denote the corresponding $n+1$-dimensional cube of additive
categories.
\begin{itemize}
\item[(a)] There is a fibration sequence
   $$S^\oplus[0\to\cc M]\to S^\oplus[\cc M'\to\cc M]\to S^\oplus[\cc M'\to 0].$$
\item[(b)] The space $S^\oplus[\cc M\bl 1\to\cc M]$ is contractible.
\item[(c)] $S^\oplus[0\to\cc M]$ is homotopy equivalent to $S^\oplus\cc M$.
\item[(d)] $S^\oplus[\cc M\to 0]=S^\oplus S^\oplus\cc M$ is is a delooping of $S^\oplus\cc M$.
\item[(f)] There is a fibration sequence $S^\oplus\cc M'\to S^\oplus\cc M\to S^\oplus[\cc M'\to\cc
M]$.
\end{itemize}
\end{lem}

Let $\cc A:(Sm/k)^{\op}\times AffSm/k\to AddCats$ be a bivariant
additive category. Given $X\in Sm/k$, $Y\in AffSm/k$ and $n>0$, the
cube of schemes $\bb G_m^{\wedge n}$ gives rise to a cube of
additive categories $\cc A(X,Y\times\bb G_m^{\wedge n})$. Its
vertices are $\cc A(X,Y\times\bb G_m^{\times\ell})$, $0\leq\ell\leq
n$. The edges of the cube are given by the natural additive functors
$i_s:\cc A(X,Y\times\bb G_m^{\times(\ell-1)})\to\cc A(X,Y\times\bb
G_m^{\times\ell})$ induced by the embeddings $i_s:\bb
G_m^{\times(\ell-1)}\to\bb G_m^{\times\ell}$ of the form
   $$(x_1,\ldots,x_{\ell-1})\longmapsto(x_1,\ldots,1,\ldots,x_{\ell-1}),$$
where 1 is the $s$th coordinate.

Thus we obtain a cube of bivariant additive categories $\cc
A\langle\bb G_m^{\wedge n}\rangle$. Grayson's $K$-theory of $\cc
A\langle\bb G_m^{\wedge n}\rangle$ produces a functor
   $$K^{Gr}(C^\oplus\cc A\langle\bb G_m^{\wedge n}\rangle):(Sm/k)^{\op}\times AffSm/k\to Sp^\Sigma,
     \quad(X,Y)\mapsto K^{Gr}(C^\oplus\cc A(X,Y\times\bb G_m^{\wedge n})).$$
Here $Sp^\Sigma$ stands for the category of symmetric spectra in the
sense of~\cite{HSS}. It is directly verified that
   \begin{equation}\label{ilin}
    K^{Gr}_0(C^\oplus\cc A(X,Y\times\bb G_m^{\wedge n}))=K^{Gr}_0(\cc A(X,Y\times\bb G_m^{\times n}))/
     \sum^n_{s=1}(i_s)_*(K^{Gr}_0(\cc A(X,Y\times\bb G_m^{\times n-1}))).
   \end{equation}
Indeed, the case $n=1$ follows from Lemma~\ref{grfiltr}(f) and the
general case is checked by induction.

\section{The category of bispectra}

In this paper we work with the category $Pre^\Sigma(Sm/k)$ of
presheaves of symmetric spectra. It has three model category
structures, each of which we discuss separately.

\begin{defs}{\rm
A morphism $f$ in $Pre^\Sigma(Sm/k)$ is a {\it stable weak
equivalence\/} (respectively {\it stable projective fibration}) if
$f(X)$ is a stable weak equivalence (respectively stable projective
fibration) in $Sp^\Sigma$ for all $X\in Sm/k$. It is a {\it stable
projective cofibration\/} if $f$ has the left lifting property with
respect to all stable projective acyclic fibrations.

}\end{defs}

Recall that the Nisnevich topology is generated by elementary
distinguished squares, i.e. pullback squares
   \begin{equation}\label{squareQ}
    \xymatrix{\ar@{}[dr] |{\textrm{$Q$}}U'\ar[r]\ar[d]&X'\ar[d]^\phi\\
              U\ar[r]_\psi&X}
   \end{equation}
where $\phi$ is etale, $\psi$ is an open embedding and
$\phi^{-1}(X\setminus U)\to(X\setminus U)$ is an isomorphism of
schemes (with the reduced structure).

\begin{defs}{\rm
(1) A stably fibrant presheaf $M\in Pre^\Sigma(Sm/k)$ is {\it
Nisnevich local\/} if for each elementary distinguished square $Q$
the square of symmetric spectra $M(Q)$ is a homotopy pullback.

(2) A Nisnevich local presheaf $M\in Pre^\Sigma(Sm/k)$ is {\it $\bb
A^1$-local\/} if the natural map
   $$M(X)\to M(X\times\bb A^1)$$
is a stable equivalence of symmetric spectra for all $X\in Sm/k$.

(3) A map $f:A\to B$ in $Pre^\Sigma(Sm/k)$ is a {\it local weak
equivalence\/} (respectively {\it motivic equivalence}) if the map
of spaces
   $$f^*:\Map(B,M)\to\Map(A,M)$$
is a weak equivalence for any Nisnevich local (respectively $\bb
A^1$-local) presheaf $M$.

(4) The {\it Nisnevich local model category\/} (respectively the
{\it motivic model category}) on pre\-sheaves of symmetric spectra,
denoted by $Pre_{nis}^\Sigma(Sm/k)$ (respectively
$Pre_{mot}^\Sigma(Sm/k)$), is determined by stable projective
cofibrations and local weak equivalences (respectively motivic
equivalences). The fibrations are defined by the corresponding
lifting property. The homotopy category of $Pre_{mot}^\Sigma(Sm/k)$
will be denoted by $SH_{S^1}(k)$.

}\end{defs}

We define the mapping cylinder $\cyl(f)$ of a map $f:A\to B$ between
cofibrant objects in a simplicial model category $\cc M$. Let
$A\otimes\varDelta^1$ denote the standard cylinder object for $A$.
One has a commutative diagram
   $$\xymatrix{A\sqcup A\ar[r]^(.55)\nabla\ar[d]_{i=i_0\sqcup i_1} &A\\
               A\otimes\varDelta^1\ar[ur]_\sigma}$$
in which $i$ is a cofibration and $\sigma$ is a weak equivalence.
Each $i_\epsilon$ must be a trivial cofibration.

Form the pushout diagram
   $$\xymatrix{
      A\ar[r]^f\ar[d]_{i_0}&B\ar[d]^{i_{0*}}\\
      A\otimes\varDelta^1\ar[r]^{f_*} &\Cyl(f).
     }$$
Then $(f\sigma)\circ i_0=f$, and so there is a unique map
$q:\Cyl(f)\to B$ such that $qf_*=f\sigma$ and $qi_{0*}=1_B$. Put
$\cyl(f)=f_*i_1$; then $f=q\circ\cyl(f)$.

If $A, B$ are cofibrant in $\cc M$, then so is $\Cyl(f)$. Observe
also that $q$ is a weak equivalence. The map $\cyl(f)$ is a
cofibration, since the diagram
   $$\xymatrix{
      A\sqcup A\ar[r]^{f\sqcup 1_A}\ar[d]_{i_0\sqcup i_1}&B\sqcup A\ar[d]^{i_{0*}\sqcup\cyl(f)}\\
      A\otimes\varDelta^1\ar[r]^{f_*} &\Cyl(f).
     }$$
is a pushout.

Consider the category $Pre(Sm/k)$ of presheaves of pointed
simplicial sets. We can define the projective model category
structure on it, where the weak equivalences and fibrations are
defined schemewise. Let $\iota:pt=\spec k\to\bb G_m$ be the
embedding $\iota(pt)=1\in\bb G_m$. The mapping cylinder yields a
factorization of the induced map
   $$\spec k_+\hookrightarrow\Cyl(\iota)\lra{\simeq}(\bb G_{m})_+$$
into a cofibration and a simplicial homotopy equivalence in
$Pre(Sm/k)$. Let $\bb G$ denote the cofibrant pointed presheaf
$\Cyl(\iota)/\spec k_+$.

Let $Pre^{\Sigma,\bb G}(Sm/k)$ denote the category of $\bb
G$-spectra in $Pre^\Sigma(Sm/k)$. Its objects are the sequences
$(X_0,X_1,\ldots)$ of presheaves of symmetric spectra $X_i$-s
together with bonding maps $X_i\to\Omega_{\bb G}X_{i+1}$, where
$\Omega_{\bb G}X_{i+1}=\underline{\Hom}(\bb G,X_{i+1})$. Morphisms
are defined levelwise and must be consistent with bonding maps. This
category will also be referred as the {\it category of $(S^1,\bb
G)$-bispectra\/} or just {\it bispectra}. We define the {\it stable
projective model structure\/} on $Pre^{\Sigma,\bb G}(Sm/k)$
(respectively the {\it Nisnevich local and motivic model
structure\/}) as the stable model category of $\bb G$-spectra in the
sense of Hovey~\cite{Hov} associated with the model category
$Pre^{\Sigma}(Sm/k)$ (respectively $Pre^{\Sigma}_{nis}(Sm/k)$ and
$Pre^{\Sigma}_{mot}(Sm/k)$). Using Hovey's notation~\cite{Hov}, it
is the model category $Sp^{\bb N}(Pre^{\Sigma}(Sm/k),\bb G\otimes-)$
(respectively $Sp^{\bb N}(Pre^{\Sigma}_{nis}(Sm/k),\bb G\otimes-)$
and $Sp^{\bb N}(Pre^{\Sigma}_{mot}(Sm/k),\bb G\otimes-)$). In what
follows we shall denote the homotopy category of $Sp^{\bb
N}(Pre^{\Sigma}_{mot}(Sm/k),\bb G\otimes-)$ by $SH(k)$. It is one of
equivalent definitions of the Voevodsky stable motivic category of
the field $k$~\cite{VoeICM}.

The main bispectrum we shall work with is produced by a bivariant
additive category
   $$\cc A:(Sm/k)^{\op}\times AffSm/k\to AddCats.$$
Namely, let
   $$A_Y=(A_Y(0),A_Y(1),A_Y(2),\ldots)$$
be the sequence of presheaves of symmetric spectra
   $$A_Y(n)=K^{Gr}(C^\oplus\cc A(-,Y\times\bb G_m^{\wedge n})),\quad n\geq 0.$$
We want to construct bonding maps
   \begin{equation*}
    a_n:A_Y(n)\to\Omega_{\bb G}A_{Y}(n+1).
   \end{equation*}
Each $a_n$ is uniquely determined by a map
   \begin{equation}\label{hook}
    \beta:K^{Gr}(C^\oplus\cc A(-,Y\times\bb G_m^{\wedge n}))\to K^{Gr}(C^\oplus\cc A(-\times\bb G_m,Y\times\bb G_m^{\wedge n+1}))
   \end{equation}
and a homotopy
   \begin{equation}\label{sorok}
    h:K^{Gr}(C^\oplus\cc A(-,Y\times\bb G_m^{\wedge n}))\to K^{Gr}(C^\oplus\cc A(-\times\spec k,Y\times\bb G_m^{\wedge n+1}))^I
   \end{equation}
such that $d_0h=\iota^*\beta$ and $d_1h$ factors trough the zero
object levelwise.

We first construct the maps $\beta$ and $h$ for $n=0$. By definition
of a bivariant additive category, there is an additive functor
   $$\alpha_{\bb G_m}:\cc A(X,Y)\to\cc A(X\times\bb G_m,Y\times\bb G_m),\quad X\in Sm/k.$$
The map $\beta$ is induced by the composition
   $$\cc A(X,Y)\xrightarrow{\alpha_{\bb G_m}}\cc A(X\times\bb G_m,Y\times\bb G_m)\bl p\to C^\oplus\cc A(X\times\bb G_m,Y\times\bb G_m^{\wedge 1}),$$
where $p$ is a natural simplicial functor of simplicial categories
(we consider $\cc A(X\times\bb G_m,Y\times\bb G_m)$ as a simplicial
category in a trivial way).

One has a commutative square of additive functors
   $$\xymatrix{\cc A(X\times\bb G_m,Y\times\bb G_m)\ar[rrr]^{\cc A(1_X\times\iota,1_{Y\times\bb G_m})}&&&\cc A(X\times\spec k,Y\times\bb G_m)\\
               \cc A(X,Y)\ar[u]_{\alpha_{\bb G_m}}\ar[rrr]^{\alpha_{\spec k}}
               &&&\cc A(X\times\spec k,Y\times\spec k).\ar[u]_{\cc A(1_{X\times\spec k},1_{\spec k}\times\iota)}}$$
On the other hand, there is a commutative diagram of simplicial
additive categories\footnotesize
    $$\xymatrix{\cc A(X\times\spec k,Y\times\bb G_m)\ar[rr]^p&&C^\oplus\cc A(X\times\spec k,Y\times\bb G_m^{\wedge 1})\\
                \cc A(X\times\spec k,Y\times\spec k)\ar[u]^{\iota_*}\ar[rr]^(.35){p'}
                &&C^\oplus[\cc A(X\times\spec k,Y\times\spec k)\xrightarrow{\id}\cc A(X\times\spec k,Y\times\spec k)]\ar[u]_{\iota_*}}$$
\normalsize Recall that the {\it path space\/} $PX$ of a simplicial
object $X:\Delta^{\op}\to\cc D$ in a category $\cc D$ is defined as
the composition of $X$ with the shift functor $P:\Delta\to\Delta$
that takes $[n]$ to $[n+1]$ (by mapping $i$ to $i+1$). The right
lower corner of the diagram can be identified with the simplicial
path space $P(S^\oplus\cc A(X\times\spec k,Y\times\spec k))$.
By~\cite[1.5.1]{Wal} there is a canonical contraction of this
simplicial set into the set of its zero simplices regarded as a
constant simplicial set. Since $P(S^\oplus\cc A(X\times\spec
k,Y\times\spec k))$ has only one zero simplex, it follows that there
is a canonical simplicial homotopy
   $$H:P(S^\oplus\cc A(X\times\spec k,Y\times\spec k))\to P(S^\oplus\cc A(X\times\spec k,Y\times\spec k))^I$$
such that $d_0H=1$ and $d_1H=const$.

Now the map $h$~\eqref{sorok} is induced by the composite map
   $$\xymatrix{&&(C^\ps\cc A(X\times\spec k,Y\times\bb G_m^{\wedge 1}))^I\\
               &P(S^\oplus\cc A(X\times\spec k,Y\times\spec k))\ar[r]^H&
               P(S^\oplus\cc A(X\times\spec k,Y\times\spec k))^I\ar[u]^{\iota_*^I}\\
               \cc A(X,Y)\ar[r]^(.32){\alpha_{\spec k}}&\cc A(X\times\spec k,Y\times\spec k)\ar[u]^{p'}}$$
Since $d_1\circ\iota_*^I\circ H=\iota_*\circ d_1\circ H=const$, then
$d_1h$ factors through the zero object. On the other hand,
   $$d_0\circ\iota_*^I\circ H\circ p'\circ\alpha_{\spec k}=\iota_*\circ d_0\circ H\circ p'\circ\alpha_{\spec k}=\iota_*\circ p'\circ\alpha_{\spec k}
     =p\circ\iota_*\circ\alpha_{\spec k}=p\circ\iota^*\circ\alpha_{\bb G_m}.$$
Moreover, $p\circ\iota^*\circ\alpha_{\bb G_m}=\iota^*\circ
p\circ\alpha_{\bb G_m}$, and therefore $d_0h=\iota^*\beta$. The
bonding map $a_0:A_Y(0)\to\Omega_{\bb G}A_Y(1)$ is now constructed.
The definition of each $a_n:A_Y(n)\to\Omega_{\bb G}A_Y(n+1)$ is
similar to that of $a_0$. The only difference is that we replace the
bivariant additive category $\cc A(X,Y)$ by the multisimplicial
bivariant additive category $C^\oplus\cc A(X,Y\times\bb G_m^{\wedge
n})$.

Given an abelian monoid $(A,+)$, denote by $EM(A)$ its
Eilenberg--Mac~Lane symmetric spectrum in the sense
of~\cite[Appendix~A]{GP}. It is defined in terms of the
$\sigma$-construction and is similar to the definition of the
Waldhausen $K$-theory spectrum that uses the
$S$-construction~\cite{Wal}. By definition, $\sigma A$ is a
simplicial set whose $n$-simplices are the functions
   $$a:\Ob\Ar[n]\to A,\quad (i,j)\mapsto a(i,j)=a_{i,j},$$
having the property that for every $j$, $a_{j,j}=0$ and
$a_{i,k}=a_{i,j}+a_{j,k}$ whenever $i\leq j\leq k$.

Given an $n$-dimensional cube of abelian monoids $M$, we define an
$n$-fold multisimplicial abelian monoid $C^\oplus M$ similar to
formula~\eqref{cube}. The only difference is that we consider
functions from objects of the corresponding posets ignoring poset
arrows. It is worth to mention that Lemma~\ref{grfiltr} is valid for
the $\sigma$-construction of cubes of abelian monoids. For this one
uses the additivity theorem for the $\sigma$-construction, just as
in~\cite[section~1.5]{Wal}. Applying the $\sigma$-construction to
the multisimplicial abelian monoid $C^\oplus M$, one gets a
symmetric spectrum $EM(C^\oplus M)$. It serves as the iterated
cofiber spectrum of the cube of Eilenberg--Mac~Lane's spectra
$EM(M)$.

Every $n$-dimensional cube of additive categories $\cc M$ gives rise
to an $n$-fold cube of abelian groups $K_0^{Gr}(\cc M)$. There is a
natural map $S^\oplus\cc M\to\sigma K_0^{Gr}(\cc M)$ induced by the
map sending an object of an additive category to its isomorphism
class in the Grothendieck group. This map induces a map of symmetric
spectra
   $$\tau:K^{Gr}(C^\oplus\cc M)\to EM(C^\oplus K_0^{Gr}(\cc M)).$$

For each $n\geq 0$ we set
   $$A_{0,Y}(n)=EM(C^\oplus K^{Gr}_0(\cc A(-,Y\times\bb G_m^{\wedge n}))).$$
Note that $\pi_0(A_{0,Y}(n))$ is computed by formula~\eqref{ilin}
and the other homotopy groups are zero. We have that the sequence of
symmetric spectra
   $$A_{0,Y}=(A_{0,Y}(0),A_{0,Y}(1),A_{0,Y}(2),\ldots)$$
forms a bispectrum, in which the bonding maps are defined like those
for the bispectrum $A_Y$. Moreover, there is a canonical map of
bispectra
   $$\tau:A_Y\to A_{0,Y}.$$
This map consists of the collection of canonical maps of symmetric
spectra
   $$\tau_n:A_Y(n)\to A_{0,Y}(n),\quad n\geq 0,$$
defined as above.

\section{The Grayson tower of bispectra}
\label{GraysonTower}

In this section we work in the framework of simplicial additive
categories $\cc M$ over contractible simplicial rings $R$. Given
such a pair $(\cc M,R)$, Grayson~\cite{Gr} constructs a tower of
spaces which is also referred to as the {\it Grayson tower}. Each
space of the Grayson tower is defined as a $K$-theory space of some
``multisimplicial additive category with commuting automorphisms"
associated with $(\cc M,R)$. In practice the Grayson tower gives
rise to a motivic spectral sequence (see~\cite{Gr,Wlk,Sus,GP}). The
Grayson tower for $(\cc M,R)$ can be extended to symmetric
spectra~\cite{GP}. We shall mostly adhere to~\cite{GP} in this
section.

In our setting the contractible ring $R$ is
   $$k[\Delta]:d\mapsto k[\Delta^d]=k[t_0,t_1,\ldots,t_d]/(t_0+t_1+\cdots+t_d-1).$$
In what follows we require
   $$d\mapsto\cc A(X\times\Delta^d,Y)$$
to be a $k[\Delta]$-linear additive category, where $\cc
A:(Sm/k)^{\op}\times AffSm/k\to AddCats$ is a bivariant additive category.

In order to make Grayson's machinery applicable to our setting,
throughout this section we work with a bivariant additive category
   $$\cc A:(Sm/k)^{\op}\times AffSm/k\to AddCats$$
satisfying the following property:

\begin{itemize}
\item[$(\ff{Aut})$] for every $X\in Sm/k,Y\in AffSm/k$ and $n>0$, the additive category $\cc A(X,Y\times\bb G_m^{\times n})$
can be identified with the additive category $\cc A(X,Y)(\bb
G_m^{\times n})$ whose objects are the tuples
$(P,\theta_1,\ldots,\theta_n)$, where $P\in\cc A(X,Y)$ and
$(\theta_1,\ldots,\theta_n)$ are commuting automorphisms of $P$.
More precisely, there is an isomorphism of additive categories (not
only an equivalence of categories)
   $$\rho_{X,Y,n}:\cc A(X,Y\times\bb G_m^{\times n})\to\cc A(X,Y)(\bb G_m^{\times n})$$
such that the diagram of functors
   $$\xymatrix{\cc A(X,Y\times\bb G_m^{\times n-1})\ar[d]_{i_s}\ar[rr]^{\rho_{X,Y,n-1}}&&\cc A(X,Y)(\bb G_m^{\times n-1})\ar[d]^{j_s}\\
               \cc A(X,Y\times\bb G_m^{\times n})\ar[rr]^{\rho_{X,Y,n}}&&\cc A(X,Y)(\bb G_m^{\times n})}$$
is commutative. Here $i_s$ (respectively $j_s$) stands for the
functor induced by the map $(x_1,\ldots,x_{n-1})\in\bb G_m^{\times
n-1}\mapsto(x_1,\ldots,1,\ldots,x_{n-1})\in\bb G_m^{\times n}$ with
1 the $s$th coordinate (respectively
$(\theta_1,\ldots,\theta_{n-1})\mapsto(\theta_1,\ldots,1,\ldots,\theta_{n-1})$).
We also require each identification $\rho_{X,Y,n}$ to be functorial
in both arguments.
\end{itemize}

We can form a cube of additive categories $\cc A(X,Y)(\bb
G_m^{\wedge n})$ whose vertices are $\cc A(X,Y)(\bb G_m^{\times
\ell})$, $\ell\leq n$, and edges are given by the functors $j_s$.
The $(\ff{Aut})$-property implies the cubes $\cc A(X,Y)(\bb
G_m^{\wedge n})$ and $\cc A(X,Y\times\bb G_m^{\wedge n})$ are
isomorphic.

Consider the map of bispectra
   $$\tau:A_Y\to A_{0,Y}.$$
For every $n\geq 0$ there is a triangle in the homotopy category
$\Ho(Pre^\Sigma(Sm/k))$ (see~\cite[section~7]{GP} for details)
   $$|d\mapsto A_Y(n+1)(-\times\Delta^d)|\xrightarrow{\gamma_n}
     \Omega|d\mapsto A_Y(n)(-\times\Delta^d)|\xrightarrow{\Omega\tau_n}
     \Omega|d\mapsto A_{0,Y}(n)(-\times\Delta^d)|.$$
The map $\gamma_1$ is induced by a zigzag map of spectra
   $$K^{Gr}(C^{\ps}\cc A(-,Y\times\bb G_m^{\wedge 1}))\lra{v}\Omega K^{Gr}(S^{-1}S\cc A(-,Y))\xleftarrow{\Omega s}\Omega K^{Gr}(\cc A(-,Y)).$$
Here $S^{-1}S$ stands for Quillen's construction (see
Section~\ref{kglsp} for more details), $s:K^{Gr}(\cc A(-,Y))\to
K^{Gr}(S^{-1}S\cc A(-,Y))$ is a stable equivalence induced by the
map sending an object $M\in\cc A(-,Y)$ to $(M,0)\in S^{-1}S\cc
A(-,Y)$ (see~\cite[9.3]{Gr} and Section~\ref{kglsp}), and $v$ is a
natural map that exists by~\cite[9.4]{Gr} and~\cite[p.~16]{Wlk}. The
map $\gamma_n$ is defined as $\gamma_1$ by replacing $\cc A(-,Y)$
with $\cc A(-,Y\times\bb G_m^{\wedge n})$. Note that each map in the
zigzag agrees with the bonding map in $\bb G$-direction.

Since the category $\Ho(Pre^{\Sigma}(Sm/k))$ is triangulated with
$\Sigma_s=-\wedge S^1$ a shift functor, the latter triangle gives a
triangle
   \begin{equation*}\label{}
    \Sigma_s|d\mapsto A_Y(n+1)(-\times\Delta^d)|\to|d\mapsto A_Y(n)(-\times\Delta^d)|
     \to|d\mapsto A_{0,Y}(n)(-\times\Delta^d)|.
   \end{equation*}
We shall also call it the {\it Grayson triangle}.

We obtain a tower in $\Ho(Pre^{\Sigma}(Sm/k))$
   \begin{equation}\label{polez}
    \cdots\to\Sigma_s^{n+1}|d\mapsto A_Y(n+1)(-\times\Delta^d)|\to\Sigma_s^n
    |d\mapsto A_Y(n)(-\times\Delta^d)|\to\cdots\to|d\mapsto A_{Y}(0)(-\times\Delta^d)|,
   \end{equation}
in which successive cones are of the form
   \begin{equation*}
    \Sigma_s^n|d\mapsto A_{0,Y}(n)(-\times\Delta^d)|.
   \end{equation*}
We shall also refer to it as the {\it Grayson tower\/} for $\cc A$.

Given $F,G\in Pre^\Sigma_{nis}(Sm/k)$ we shall use the following
notation:
   $$[F,G]:=\Hom_{\Ho Pre^\Sigma_{nis}(Sm/k)}(F,G).$$
Given $X\in Sm/k$, $Y\in Affsm/k$ and $p,q\in\bb Z$, we also set
   $$H_{\cc A}^{p,q}(X,Y):=[X_+,\Sigma_s^{p-q}|d\mapsto A_{0,Y}(q)(-\times\Delta^d)|]$$
and\label{KpA}
   $$K_i^{\cc A}(X,Y):=[X_+,\Sigma_s^{-i}|d\mapsto A_{Y}(0)(-\times\Delta^d)|],\quad i\in\bb Z.$$

\begin{rem}\label{isp}{\rm
Let $D(NSh)$ be the derived category of Nisnevich sheaves of abelian
groups on $Sm/k$. The Dold--Kan correspondence yields a complex of
presheaves of abelian groups
   $$C^*(A_{0,Y}(q))$$
which uniquely corresponds to the simplicial presheaf
   $$d\mapsto K_0^{Gr}(C^\oplus\cc A(-\times\Delta^d,Y\times\bb G_m^{\wedge q})).$$
After sheafifying $C^*(A_{0,Y}(q))$ degreewise in the Nisnevich
topology, we get a bounded above complex $C^*(A_{0,Y}(q))_{\nis}\in
D(NSh)$ (the indexing is cohomological). It is then proved similar
to~\cite[7.8]{GP} that
    $$H_{\cc A}^{p,q}(X,Y)=H^p_{\nis}(X,C^*(A_{0,Y}(q))_{\nis}[-q]),$$
where the right hand side stands for Nisnevich hypercohomology
groups of $X$ with coeffitients in $C^*(A_{0,Y}(q))_{\nis}[-q]$ (the
shift is cohomological).

}\end{rem}

\begin{thm}[Grayson]\label{grspektr}
The Grayson tower~\eqref{polez} produces a strongly convergent
spectral sequence
   \begin{equation}\label{grspseq}
    E_2^{pq}=H^{p-q,-q}_{\cc A}(X,Y)\Longrightarrow K_{-p-q}^{\cc A}(X,Y),\quad X\in Sm/k,Y\in Affsm/k,
   \end{equation}
which will also be referred to as the \emph{Grayson spectral
sequence\/} for $\cc A$.
\end{thm}

\begin{proof}
This is proved similar to~\cite[7.9]{GP}.
\end{proof}

\begin{cor}\label{syak}
If the groups $H_{\cc A}^{p,q}(X,Y)$ are homotopy invariant in the
first argument, then so are the groups $K_i^{\cc A}(X,Y)$-s. In
particular, every fibrant replacement of $|d\mapsto
A_{Y}(0)(-\times\Delta^d)|$ in the Nisnevich local model category
$Pre_{nis}^\Sigma(Sm/k)$ is fibrant in the motivic model category
$Pre_{mot}^\Sigma(Sm/k)$.
\end{cor}

Below we shall study conditions when the Grayson spectral
sequence~\eqref{grspseq} is expressed in terms of bispectra.

Given a bispectrum $X=(X_0,X_1,\ldots)$, the {\it shift in $\bb
G$-direction\/} $\Sigma_{\bb G}X$ is the bispectrum
$(X_1,X_2,\ldots)$. Similarly, the $n$th shift $\Sigma_{\bb G}^nX$
is the bispectrum $(X_n,X_{n+1},\ldots)$. For each $n\geq 0$, we
have a triangle in the homotopy category of bispectra
$\Ho(Pre^{\Sigma,\bb G}(Sm/k))$
   $$\Sigma^{n+1}_{\bb G}|d\mapsto A_Y(-\times\Delta^d)|\xrightarrow\gamma\Omega\Sigma^n_{\bb G}|d\mapsto A_Y(-\times\Delta^d)|
     \xrightarrow{\Omega\tau}\Omega\Sigma^n_{\bb G}|d\mapsto A_{0,Y}(-\times\Delta^d)|.$$
Since the category $\Ho(Pre^{\Sigma,\bb G}(Sm/k))$ is triangulated
with $\Sigma_s=-\wedge S^1$ a shift functor, the latter triangle
gives a triangle
   \begin{equation*}\label{}
    \Sigma_s\Sigma^{n+1}_{\bb G}|d\mapsto A_Y(-\times\Delta^d)|\to\Sigma^n_{\bb G}|d\mapsto A_Y(-\times\Delta^d)|
     \xrightarrow{\tau}\Sigma^n_{\bb G}|d\mapsto A_{0,Y}(-\times\Delta^d)|.
   \end{equation*}
We shall also call it the {\it Grayson triangle of bispectra}.

We obtain a tower of bispectra in $\Ho(Pre^{\Sigma,\bb G}(Sm/k))$
   \begin{equation}\label{sber}
    \cdots\to\Sigma_s^{n+1}\Sigma^{n+1}_{\bb G}|d\mapsto A_Y(-\times\Delta^d)|\to\Sigma_s^n\Sigma^n_{\bb
    G}|d\mapsto A_Y(-\times\Delta^d)|\to\cdots\to|d\mapsto A_Y(-\times\Delta^d)|,
   \end{equation}
in which successive cones are of the form
   \begin{equation*}
    \Sigma_s^n\Sigma^n_{\bb G}|d\mapsto A_{0,Y}(-\times\Delta^d)|.
   \end{equation*}
We shall also refer to it as the {\it Grayson tower of bispectra\/}
for $\cc A$.

Given a presheaf $\cc F$ of abelian groups and a scheme $X\in Sm/k$,
one sets
   $$\cc F(X\wedge\bb G_m):=\kr(\iota^*:\cc F(X\times\bb G_m)\to\cc F(X\times\spec k)).$$
There is a map of complexes
   $$\beta:C^*(A_{0,Y}(q))\to C^*(A_{0,Y}(q+1))(-\times\bb G_m),$$
where the left arrow is induced by map~\eqref{hook}.
Homotopy~\eqref{sorok} implies $\beta$ uniquely factors through
$C^*(A_{0,Y}(q+1))(-\wedge\bb G_m)$. Therefore one gets maps
   $$\beta^{p,q}:H^{p,q}_{\cc A}(X,Y)\to H^{p+1,q+1}_{\cc A}(X\wedge\bb G_m,Y).$$

\begin{defs}{\rm
We say that the bigraded presheaves $H^{*,*}_{\cc A}(-,Y)$ satisfy
the {\it cancelation property\/} if all maps $\beta^{p,q}$ are
isomorphisms.

}\end{defs}

Let $X=(X_0,X_1,\ldots)$ and $Y=(Y_0,Y_1,\ldots)$ be two bispectra.
Recall that a map of bispectra $f:X\to Y$ is a level Nisnevich local
equivalence if so is each $f_n:X_n\to Y_n$ in
$Pre^\Sigma_{nis}(Sm/k)$. By common facts for model categories (see,
e.g.,~\cite{Hov}) there is a level Nisnevich local equivalence of
bispectra
   $$u:X\to\wt X$$
such that each map $u_n:X_n\to\wt X_n$ is a cofibration and each
$\wt X_n$ is fibrant in $Pre^\Sigma_{nis}(Sm/k)$. Moreover, the map
is functorial in $X$.

Consider the bispectrum $A_{0,Y}=(A_{0,Y}(0),A_{0,Y}(1),\ldots)$.
Denote by $A_{0,Y}^\Delta$ and $\wt A_{0,Y}^\Delta$ the bispectra
   $$(|d\mapsto A_{0,Y}(0)(-\times\Delta^d)|,|d\mapsto A_{0,Y}(1)(-\times\Delta^d)|,\ldots)$$
and
   $$(\wt{(A_{0,Y}^\Delta)}_0,\wt{(A_{0,Y}^\Delta)}_1,\wt{(A_{0,Y}^\Delta)}_2,\ldots)$$
respectively. Then there is a map of bispectra (see above)
   $$u:A_{0,Y}^\Delta\to\wt A_{0,Y}^\Delta.$$
Observe that there is an isomorphism
   $$H^{p,q}_{\cc A}(X,Y)\cong\pi(X_+,\Sigma^{p-q}_s(\wt A_{0,Y}^\Delta)_q),$$
where the right hand side stands for the usual homotopy equivalence
classes of maps.

\begin{lem}\label{brussels}
The bigraded presheaves $H^{*,*}_{\cc A}(-,Y)$ satisfy the
cancelation property if and only if each structure map of the
bispectrum $\wt A_{0,Y}^\Delta$
   $$(\wt A_{0,Y}^\Delta)_n\to\Omega_{\bb G}(\wt A_{0,Y}^\Delta)_{n+1}$$
is a stable weak equivalence in $Pre^\Sigma(Sm/k)$.
\end{lem}

We can also define bispectra $A_{Y}^\Delta$ and $\wt A_{Y}^\Delta$
similar to the bispectra $A_{0,Y}^\Delta$ and $\wt A_{0,Y}^\Delta$.
There is a map of towers in $\Ho(Pre^\Sigma(Sm/k))$
   $$\xymatrix{\cdots\ar[r]&\Sigma_s^{n+1}(A_{Y}^\Delta)_{n+1}\ar[r]\ar[d]&\Sigma_s^{n}(A_{Y}^\Delta)_n\ar[d]\ar[r]&\cdots\ar[r]&(A_{Y}^\Delta)_0\ar[d]\\
               \cdots\ar[r]&\Sigma_s^{n+1}(\wt A_{Y}^\Delta)_{n+1}\ar[r]&\Sigma_s^{n}(\wt A_{Y}^\Delta)_n\ar[r]&\cdots\ar[r]&(\wt A_{Y}^\Delta)_0,}$$
where the upper tower is the Grayson tower and the vertical
morphisms are Nisnevich local equivalences in
$Pre^\Sigma_{nis}(Sm/k)$. Moreover, we have maps of successive cones
of both towers
   $$\Sigma_s^{n}u_n:\Sigma_s^{n}(A_{0,Y}^\Delta)_n\to\Sigma_s^{n}(\wt A_{0,Y}^\Delta)_n.$$

For every $q\geq 0$ and $p\in\bb Z$ one sets
   $$K_{-p}^{\cc A}(X,Y\wedge\bb G_m^{q}):=[X_+,\Sigma_s^{p-q}(A_{Y}^\Delta)_q]$$
Observe that there is an isomorphism
   $$K_{-p}^{\cc A}(X,Y\wedge\bb G_m^{q})\cong\pi(X_+,\Sigma_s^{p-q}(\wt A_{Y}^\Delta)_q).$$

\begin{thm}[Cancelation for $K$-theory]\label{novgorod}
Suppose the bigraded presheaves $H^{*,*}_{\cc A}(-,Y)$ satisfy the
cancelation property. Then each structure map of the bispectrum $\wt
A_{Y}^\Delta$
   \begin{equation}\label{len}
    (\wt A_{Y}^\Delta)_q\to\Omega_{\bb G}(\wt A_{Y}^\Delta)_{q+1},\quad q\geq 0,
   \end{equation}
is a stable weak equivalence of symmetric spectra. In particular,
the natural map
   $$K_{-p}^{\cc A}(X,Y\wedge\bb G_m^{q})\to K_{-p-1}^{\cc A}(X\wedge\bb G_m,Y\wedge\bb G_m^{q+1}),$$
induced by~\eqref{len}, is an isomorphism for all $p\in\bb Z$.
\end{thm}

\begin{proof}
We have a map of towers in $\Ho(Pre^\Sigma(Sm/k))$
   $$\xymatrix{\cdots\ar[r]&\Sigma_s^{n+1}(\wt A_{Y}^\Delta)_{n+1}\ar[r]\ar[d]&\Sigma_s^{n}(\wt A_{Y}^\Delta)_n\ar[d]\ar[r]&\cdots\ar[r]&(\wt A_{Y}^\Delta)_0\ar[d]\\
               \cdots\ar[r]&\Sigma_s^{n+1}\Omega_{\bb G}(\wt A_{Y}^\Delta)_{n+2}\ar[r]&\Sigma_s^{n}\Omega_{\bb G}(\wt A_{Y}^\Delta)_{n+1}
               \ar[r]&\cdots\ar[r]&\Omega_{\bb G}(\wt A_{Y}^\Delta)_1,}$$
where the upper tower is the Grayson tower and the vertical
morphisms are structure morphisms of the bispectrum $\wt
A_Y^\Delta$. By Lemma~\ref{brussels} all maps of successive cones
   $$\Sigma_s^{n}(\wt A_{0,Y}^\Delta)_n\to\Sigma_s^{n}\Omega_{\bb G}(\wt A_{0,Y}^\Delta)_{n+1}$$
are stable weak equivalences in $Pre^\Sigma(Sm/k)$.

For every $X\in Sm/k$, Theorem~\ref{grspektr} implies the upper
tower produces a strongly convergent spectral sequence
   $$E^2_{pq}=\pi_{p+q}(\Sigma_s^{q}(\wt A_{0,Y}^\Delta)_q(X))\Longrightarrow\pi_{p+q}((\wt A_{Y}^\Delta)_0(X))$$
and the lower tower produces a strongly convergent spectral sequence
   $$E^2_{pq}=\pi_{p+q}(\Sigma_s^{q}\Omega_{\bb G}(\wt A_{0,Y}^\Delta)_{q+1}(X))\Longrightarrow\pi_{p+q}(\Omega_{\bb G}(\wt A_{Y}^\Delta)_1(X)).$$
Since both spectral sequences are isomorphic, we conclude that the
map
   $$(A_{Y}^\Delta)_0\to\Omega_{\bb G}(A_{Y}^\Delta)_1$$
is a stable weak equivalence in $Pre^\Sigma(Sm/k)$. It is proved in
a similar fashion that all other vertical maps in the diagram above
are stable weak equivalences in $Pre^\Sigma(Sm/k)$, and hence so are
their desuspensions.
\end{proof}

We are now in a position to prove the main result of the section.

\begin{thm}\label{horosho}
Suppose the groups $H^{p,q}_{\cc A}(X,Y)$ are homotopy invariant in
the first argument and satisfy the cancelation property. Then the
bispectra $\wt A_{Y}^\Delta$ and $\wt A_{0,Y}^\Delta$ are
motivically fibrant in $Pre^{\Sigma,\bb G}_{mot}(Sm/k)$ and the maps
   $$A_{Y}\to\wt A_{Y}^\Delta,\quad A_{0,Y}\to\wt A_{0,Y}^\Delta$$
are motivic weak equivalences of bispectra. As a result, one has a
tower in $SH(k)$
   \begin{equation}\label{russa}
    \cdots\to\Sigma_s^{q+1}\Sigma^{q+1}_{\bb G}A_Y\to\Sigma_s^q\Sigma^q_{\bb G}A_Y\to\cdots\to A_Y,
   \end{equation}
in which successive cones are of the form $\Sigma_s^q\Sigma^q_{\bb
G}A_{0,Y}$. This tower produces a strongly convergent spectral
sequence
   $$E^2_{pq}=SH(k)(X_+,\Sigma_s^{-p}\Sigma^q_{\bb G}A_{0,Y})\Longrightarrow SH(k)(X_+,\Sigma_s^{-p-q}A_Y),$$
which is isomorphic to the Gayson spectral sequence for $\cc A$.
\end{thm}

\begin{proof}
By Corollary~\ref{syak} $(\wt A_{Y}^\Delta)_0$ is fibrant in
$Pre_{mot}^\Sigma(Sm/k)$. It is proved in a similar way that each
$(\wt A_{Y}^\Delta)_q$, $q>0$, is fibrant in
$Pre_{mot}^\Sigma(Sm/k)$. So the bispectrum $\wt A_{Y}^\Delta$ is
level motivically fibrant. Theorem~\ref{novgorod} implies it is
fibrant in $Pre_{mot}^{\Sigma,\bb G}(Sm/k)$. The fact that the
bispectrum $\wt A_{0,Y}^\Delta$ is fibrant in $Pre_{mot}^{\Sigma,\bb
G}(Sm/k)$ follows from Lemma~\ref{brussels} and the homotopy
invariance of the groups $H^{p,q}_{\cc A}(X,Y)$ in the first
argument.

Let us show that $A_{Y}\to\wt A_{Y}^\Delta$ is a motivic weak
equivalence of bispectra. This map is the composition
   $$A_{Y}\to A_Y^\Delta\to\wt A_{Y}^\Delta,$$
where the right arrow is a Nisnevich local equivalence, hence a
motivic weak equivalence. The left arrow is a motivic weak
equivalence by~\cite[3.8]{MV}. Similarly, $A_{0,Y}\to\wt
A_{0,Y}^\Delta$ is a motivic weak equivalence of bispectra.

The Grayson tower of bispectra~\eqref{sber} in $\Ho(Pre^{\Sigma,\bb
G})(Sm/k)$ yields a tower of bispectra in $SH(k)$
   \begin{equation}\label{neploho}
    \cdots\to\Sigma_s^{q+1}\Sigma^{q+1}_{\bb G}A_Y\to\Sigma_s^q\Sigma^q_{\bb G}A_Y\to\cdots\to A_Y,
   \end{equation}
in which successive cones are of the form $\Sigma_s^q\Sigma^q_{\bb
G}A_{0,Y}$. This tower produces a spectral sequence
   $$E^2_{pq}=SH(k)(X_+,\Sigma_s^{-p}\Sigma^q_{\bb G}A_{0,Y})\Longrightarrow SH(k)(X_+,\Sigma_s^{-p-q}A_Y),$$
which is the same as the spectral sequence
   $$E^2_{pq}=SH(k)(X_+,\Sigma_s^{-p}\Sigma^q_{\bb G}\wt A^\Delta_{0,Y})
     \Longrightarrow SH(k)(X_+,\Sigma_s^{-p-q}\wt A^\Delta_Y).$$
Since all bispectra involved into the latter spectral sequence are
motivically fibrant, it follows that it is isomorphic to the
spectral sequence
   $$E^2_{pq}=\Hom_{\Ho(Pre^\Sigma(Sm/k))}(X_+,\Sigma_s^{-p}(\wt A^\Delta_{0,Y})_q)
     \Longrightarrow\Hom_{\Ho(Pre^\Sigma(Sm/k))}(X_+,\Sigma_s^{-p-q}(\wt A^\Delta_Y)_0).$$
It is plainly isomorphic to the Grayson spectral
sequence~\eqref{grspseq}.
\end{proof}

\section{Postnikov towers in $SH_{S^1}(k)$ and $SH(k)$}

Voevodsky~\cite{VoeAppr} has defined a canonical tower on the
motivic stable homotopy category $SH_{S^1}(k)$, which is called the
{\it motivic Postnikov tower}.

Let $\Sigma_{\bb G}^nSH_{S^1}(k)$ be the localizing subcategory of
$SH_{S^1}(k)$ generated by objects of the form $\Sigma_{\bb G}^nE$,
$E\in SH_{S^1}(k)$. This gives us the tower of localizing
subcategories
   $$\cdots\subset\Sigma_{\bb G}^{n+1}SH_{S^1}(k)\subset\Sigma^n_{\bb G}SH_{S^1}(k)\subset\cdots\subset SH_{S^1}(k).$$

Take $E\in SH_{S^1}(k)$ and consider the cohomological functor
   $$\Hom_{\Sigma^n_{\bb G}SH_{S^1}(k)}(-,E):\Sigma^n_{\bb G}SH_{S^1}(k)\to Ab.$$
By~\cite{N} this functor is represented by an object $r_nE$ of
$\Sigma^n_{\bb G}SH_{S^1}(k)$. Sending $E$ to $r_nE$ defines a right
adjoint $r_n:SH_{S^1}(k)\to\Sigma^n_{\bb G}SH_{S^1}(k)$ to the
inclusion $i_n:\Sigma^n_{\bb G}SH_{S^1}(k)\to SH_{S^1}(k)$. Let
$f_n:=i_n\circ r_n$ with counit $f_n\to\id$. Thus, for each $E\in
SH_{S^1}(k)$, there is a canonical tower in $SH_{S^1}(k)$
   $$\cdots\to f_{n+1}E\to f_nE\to\cdots\to f_0E=E,$$
the {\it motivic Postnikov tower for $S^1$-spectra}. We write
$f_{n/n+r}E$ for the cofibre of $f_{n+r}\to f_nE$; we use the
notation $s_n:=f_{n/n+1}$ to denote the $n$th slice in the Postnikov
tower.

Let $EM:Ab\to Sp^\Sigma$ be the Eilenberg--Mac~Lane functor in the
sense of~\cite[Appendix~A]{GP} and let $\bb Z(n)$, $n\geq 0$, be the
Suslin--Voevodsky complexes~\cite{SV1}.

\begin{prop}\label{kahn}
Suppose $k$ is a perfect field. Then for every $n\geq 0$ we have
$EM(\bb Z(n))=s_n(EM(\bb Z(n)))$.
\end{prop}

\begin{proof}
By~\cite[1.4.3]{KL} $EM(\bb Z(n))\in\Sigma^n_{\bb G}SH_{S^1}(k)$.
Cancellation Theorem of Voevodsky~\cite{Voe3}
and~\cite[4.2]{VoeAppr} imply
   $$\Omega^{n+1}_{\bb G}(EM(\bb Z(n))_f)\cong\Omega_{\bb G}(EM(\bb Z(0))_f)=0,$$
where $EM(\bb Z(n))_f$ is a local fibrant replacement of $EM(\bb
Z(n))$. It follows that $EM(\bb Z(n))$ is right orthogonal to
$\Sigma^{n+1}_{\bb G}SH_{S^1}(k)$, and hence $EM(\bb
Z(n))=s_n(EM(\bb Z(n)))$.
\end{proof}

Given an integer $\ell$, call $E\in SH_{S^1}(k)$ {\it
$\ell$-connected\/} if for each $n\leq\ell$, the Nisnevich sheaf
$\pi_n(\wt E)$ is zero, where $\wt E$ is a fibrant model for $E$ in
$Pre^\Sigma_{mot}(Sm/k)$. We shall also refer to $(-1)$-connected
spectra as {\it connected}.

\begin{lem}\label{nesv}
Suppose $k$ is a perfect field and $\cc A$ is a bivariant additive
category with $\ff{(Aut)}$-property. Suppose further that each
$A_{0,Y}(n)\in\Sigma^n_{\bb G}SH_{S^1}(k)$. If the presheaves
$H^{p,q}_{\cc A}(-,Y)$ are homotopy invariant, then
$A_Y(n)\in\Sigma^n_{\bb G}SH_{S^1}(k)$ for every $n\geq 0$.
\end{lem}

\begin{proof}
By Corollary~\ref{syak} $(\wt A_{Y}^\Delta)_0$ is fibrant in
$Pre_{mot}^\Sigma(Sm/k)$. It is proved in a similar way that each
$(\wt A_{Y}^\Delta)_n$, $n>0$, is fibrant in
$Pre_{mot}^\Sigma(Sm/k)$. The map $A_{Y}(n)\to(\wt A_{Y}^\Delta)_n$
is a motivic weak equivalence. It is the composition
   $$A_{Y}(n)\to(A_Y^\Delta)_n\to(\wt A_{Y}^\Delta)_n,$$
where the right arrow is a Nisnevich local equivalence, hence a
motivic weak equivalence. The left arrow is a motivic weak
equivalence by~\cite[3.8]{MV}.

We have $A_{0,Y}(m)\cong f_n(A_{0,Y}(m))\in\Sigma^n_{\bb
G}SH_{S^1}(k)$ for all $m\geq n$. Applying the functor $f_n$ to the
Grayson tower~\eqref{polez}, we get a map of towers of motivically
fibrant presheaves of spectra
   $$\xymatrix{\cdots\ar[r]&\Sigma_s^{n+1}f_{n}((\wt A_{Y}^\Delta)_{n+1})\ar[r]\ar[d]&\Sigma_s^{n}f_n((\wt A_{Y}^\Delta)_n)\ar[d]\\
               \cdots\ar[r]&\Sigma_s^{n+1}(\wt A_{Y}^\Delta)_{n+1}\ar[r]&\Sigma_s^{n}(\wt A_{Y}^\Delta)_n,}$$
in which successive cones are of the form $\Sigma_s^{n+q}(\wt
A_{0,Y}^\Delta)_{n+q}$.

Every spectrum $\Sigma_s^{n+q}(\wt A_{Y}^\Delta)_{n+q}$ is
$(n+q-1)$-connected. If $X\in Sm/k$ is of dimension $d$, then
$\Sigma_s^{n+q}(\wt A_{Y}^\Delta)_{n+q}(X)$ is $(n+q-d-1)$-connected
in $Sp^\Sigma$ by~\cite[4.3.1]{Mor}. By~\cite[6.1.1]{FS} the lower
tower produces a strongly convergent spectral sequence
   $$E^2_{pq}=\pi_{p+q}(\Sigma_s^{n+q}(\wt A_{0,Y}^\Delta)_{n+q}(X))\Longrightarrow\pi_{p+q}((\wt A_{Y}^\Delta)_n(X)).$$
To show that the spectral sequence produced by the upper tower
   $$E^2_{pq}=\pi_{p+q}(\Sigma_s^{n+q}(\wt A_{0,Y}^\Delta)_{n+q}(X))\Longrightarrow\pi_{p+q}(f_n((\wt A_{Y}^\Delta)_n)(X))$$
is strongly convergent, we need to know that each $f_n((\wt
A_{Y}^\Delta)_{n+q})$ has the same connectivity properties as $(\wt
A_{Y}^\Delta)_{n+q}$. Since $\Sigma_s^{n+q}(\wt A_{Y}^\Delta)_{n+q}$
is $(n+q-1)$-connected, it follows from~\cite[3.2]{LevGW} and
Proposition~\ref{gmloops} that so is $\Sigma_s^{n+q}f_n((\wt
A_{Y}^\Delta)_{n+q})=f_n(\Sigma_s^{n+q}(\wt A_{Y}^\Delta)_{n+q})$.
If $X\in Sm/k$ is of dimension $d$, then $\Sigma_s^{n+q}f_n((\wt
A_{Y}^\Delta)_{n+q})(X)$ is $(n+q-d-1)$-connected in $Sp^\Sigma$
by~\cite[4.3.1]{Mor}.

It follows from~\cite[6.1.1]{FS} that the second spectral sequence
is strongly convergent. We conclude that the map in $\Ho(Sp^\Sigma)$
   $$f_n((\wt A_{Y}^\Delta)_n)(X)\to(\wt A_{Y}^\Delta)_n(X)$$
is an isomorphism. We see that $A_Y(n)$ is isomorphic in
$SH_{S^1}(k)$ to $f_n((\wt A_{Y}^\Delta)_n)\in\Sigma_{\bb
G}^nSH_{S^1}(k)$.
\end{proof}

The following theorem says that Grayson's tower of
$S^1$-spectra~\eqref{polez} is isomorphic in $SH_{S^1}(k)$ to the
motivic Postnikov tower of the $K$-theory $S^1$-spectrum $K^{Gr}(\cc
A(-,Y))$. In Theorem~\ref{glavnaya} we shall extend this result to
bispectra.

\begin{thm}\label{grslice}
Suppose $k$ is a perfect field and $\cc A$ is a bivariant additive
category with $\ff{(Aut)}$-property. Suppose further that each
$A_{0,Y}(q)=s_q(A_{0,Y}(q))$. If the presheaves $H^{p,q}_{\cc
A}(-,Y)$ are homotopy invariant, then the Grayson
tower~\eqref{polez} is isomorphic in $SH_{S^1}(k)$ to the motivic
Postnikov tower
    $$\cdots\to f_{q+1}(K^{Gr}(\cc A(-,Y)))\to f_q(K^{Gr}(\cc A(-,Y)))\to\cdots\to f_0(K^{Gr}(\cc A(-,Y)))=K^{Gr}(\cc A(-,Y)).$$
\end{thm}

\begin{proof}
We have $f_0(K^{Gr}(\cc A(-,Y)))=K^{Gr}(\cc A(-,Y))$. Suppose an
isomorphism $\theta_q:\Sigma_s^q A_{Y}^\Delta(q)\cong f_q(K^{Gr}(\cc
A(-,Y)))$, $q\geq 0$, is constructed. Since
$\Sigma_s^{q+1}A_{Y}(q+1)\in\Sigma^{q+1}_{\bb G}SH_{S^1}(k)$ by the
preceding lemma and $s_q(K^{Gr}(\cc A(-,Y)))$ is right orthogonal to
$\Sigma^{q+1}_{\bb G}SH_{S^1}(k)$, it follows that there is a unique
morphism
   $$\theta_{q+1}:\Sigma_s^{q+1}A_{Y}(q+1)\to f_{q+1}(K^{Gr}(\cc A(-,Y)))$$
making the diagram
   $$\xymatrix{\Sigma_s^{q+1}A_{Y}(q+1)\ar[r]\ar[d]_{\theta_{q+1}}&\Sigma_s^{q}A_{Y}(q)\ar[d]^{\theta_q}\\
               f_{q+1}(K^{Gr}(\cc A(-,Y)))\ar[r]&f_q(K^{Gr}(\cc A(-,Y)))}$$
commutative. We claim that $\theta_{q+1}$ is an isomorphism in
$SH_{S^1}(k)$.

By assumption
$f_{q+1}(\Sigma_s^{q}A_{0,Y}(q))=f_{q+1}(\Sigma_s^{q}s_q(A_{0,Y}(q)))=\Sigma_s^{q}f_{q+1}s_q(A_{0,Y}(q))=0$.
We also have that $f_{q+1}s_q(K^{Gr}(\cc A(-,Y)))=0$, and hence the
horizontal arrows of the commutative diagram
      $$\xymatrix{f_{q+1}(\Sigma_s^{q+1}A_{Y}(q+1))\ar[r]\ar[d]_{f_{q+1}(\theta_{q+1})}
               &f_{q+1}(\Sigma_s^qA_{Y}(q))\ar[d]^{f_{q+1}(\theta_q)}\\
               f_{q+1}(f_{q+1}(K^{Gr}(\cc A(-,Y))))\ar[r]&f_{q+1}(f_q(K^{Gr}(\cc A(-,Y))))}$$
are isomorphisms. But $f_{q+1}(\theta_q)$ is an isomorphism, and
hence so is $f_{q+1}(\theta_{q+1})$. By the previous lemma
$\Sigma_s^{q+1}A_{Y}(q+1)$ is in $\Sigma^{q+1}_{\bb G}SH_{S^1}(k)$.
Since $f_{q+1}(K^{Gr}(\cc A(-,Y)))$ belongs to $\Sigma^{q+1}_{\bb
G}SH_{S^1}(k)$ as well and $f_{q+1}(\theta_{q+1})$ is an
isomorphism, we conclude that $\theta_{q+1}$ is an isomorphism.
\end{proof}

The next result computes the slices of the $K$-theory $S^1$-spectrum
$K^{Gr}(\cc A(-,Y))$. It will be extended to bispectra in
Theorem~\ref{voevslices}.

\begin{thm}\label{grslices}
Under the assumptions of Theorem~\ref{grslice} there are
isomorphisms in $SH_{S^1}(k)$
   $$s_q(K^{Gr}(\cc A(-,Y)))\cong\Sigma_s^qA_{0,Y}(q),\quad q\geq 0.$$
\end{thm}

\begin{proof}
The proof of the previous theorem shows that there is a commutative
diagram in $SH_{S^1}(k)$\footnotesize
   $$\xymatrix{\Sigma_s^{q+1}A_{Y}(q+1)\ar[r]\ar[d]_{\theta_{q+1}}&\Sigma_s^qA_{Y}(q)\ar[d]^{\theta_q}
               \ar[r]&\Sigma_s^qA_{0,Y}(q)\ar[r]&\Sigma_s^{q+2}A_{0,Y}(q+1)\ar[d]^{\Sigma_s\theta_{q+1}}\\
               f_{q+1}(K^{Gr}(\cc A(-,Y)))\ar[r]&f_q(K^{Gr}(\cc A(-,Y)))\ar[r]&s_q(K^{Gr}(\cc A(-,Y)))\ar[r]&\Sigma_sf_{q+1}(K^{Gr}(\cc A(-,Y))),}$$
\normalsize where the vertical arrows are isomorphisms. Since
$SH_{S^1}(k)$ is triangulated, then there exists an isomorphism
   $$\Sigma_s^qA_{0,Y}(q)\cong s_q(K^{Gr}(\cc A(-,Y))),$$
as required.
\end{proof}

Voevodsky~\cite{VoeProbl} defines the slice filtration in $SH(k)$
just as it is defined in $SH_{S^1}(k)$. Let $SH^{eff}(k)$ be the
smallest localizing subcategory of $SH(k)$ containing all suspension
spectra $\Sigma^\infty_{\bb G}\Sigma_s^\infty X_+$ with $X\in Sm/k$;
this is the same as the smallest localizing subcategory containing
all the $\bb G$-suspension spectra $\Sigma^\infty_{\bb G}E$ for
$E\in SH_{S^1}(k)$. For each integer $p$, let $\Sigma_{\bb G}^p
SH^{eff}(k)$ denote the smallest localizing subcategory of $SH(k)$
containing the $\bb G$-spectra $\Sigma_{\bb G}^p\cc E$ for $\cc E\in
SH^{eff}(k)$. The inclusion $i_p:\Sigma_{\bb G}^p SH^{eff}(k)\to
SH(k)$ admits the right adjoint $r_p:SH(k)\to\Sigma_{\bb G}^p
SH^{eff}(k)$; setting $f_p:=i_p\circ r_p$, one has for each $\cc
E\in SH(k)$ the functorial {\it slice tower} 
   $$\cdots\to f_{d+1}\cc E\to f_d\cc E\to\cdots\to f_0\cc E\to f_{-1}\cc E\to\cdots\to\cc E.$$
As for the slice tower in $SH_{S^1}(k)$, the existence of the
adjoint follows from~\cite{N} and the map $f_p\cc E\to\cc E$ is
universal for maps $\cc F\to\cc E$, $\cc F\in\Sigma_{\bb G}^p
SH^{eff}(k)$. The cofiber of $f_{d+1}\cc E\to f_d\cc E$ is denoted
by $s_d\cc E$.

\begin{lem}\label{babah}
Under the assumptions of Lemma~\ref{nesv} the bispectra $A_Y$ and
$A_{0,Y}$ belong to $SH^{eff}(k)$.
\end{lem}

\begin{proof}
We prove the assertion for $A_Y$, because the same arguments will
hold for $A_{0,Y}$. It is shown similar to~\cite[A.33]{PPO} that
every bispectrum $\cc E$ is the colimit of a natural sequence
   $$Tr_0\cc E\to Tr_1\cc E\to\cdots,$$
where $Tr_i\cc E$ stands fort the $i$th truncation. Moreover,
$Tr_i\cc E$ is naturally stably equivalent to $\Omega_{\bb
G}^i((\Sigma_{\bb G}^\infty E_i)^f)$ (``$f$" for fibrant
resolution).

Lemma~\ref{nesv} implies $(\Sigma_{\bb G}^\infty
A_Y(i))^f\in\Sigma^i_{\bb G}SH^{eff}(k)$, and hence $\Omega^i_{\bb
G}((\Sigma_{\bb G}^\infty A_Y(i))^f)\in SH^{eff}(k)$. We see that
$A_Y$ is the colimit of the sequence
   $$Tr_0 A_Y\to Tr_1 A_Y\to\cdots,$$
where each $Tr_i A_Y$ is in $SH^{eff}(k)$. It follows
from~\cite[A.34]{PPO} that $A_Y$ is isomorphic in $SH(k)$ to the
homotopy colimit of $Tr_i A_Y$-s, which is in $SH^{eff}(k)$. We
conclude that $A_Y\in SH^{eff}(k)$.
\end{proof}

\section{The bispectrum $KGL_{\cc A,Y}$}\label{kglsp}

Given an additive category $\cc M$, Quillen defines a new category
$S^{-1}S\cc M$ whose objects are pairs $(A,B)$ of objects of $\cc
M$. A morphism $(A,B)\to(C,D)$ in $S^{-1}S\cc M$ is given by a pair
of split monomorphisms
   $$f:A\bl{\twoheadleftarrow}\rightarrowtail C,\quad g:B\bl\twoheadleftarrow\rightarrowtail D$$
together with an isomorphism $h:\coker f\to\coker g$. By a split
monomorphism we mean a monomorphism together with a chosen
splitting. The nerve of the category $S^{-1}S\cc M$ which is also
denoted by $S^{-1}S\cc M$ is homotopy equivalent to Quillen's
$K$-theory space of $\cc M$ by~\cite{Gr1}.

The set $S^{-1}S_k\cc M$ of $k$-simplices of the category
$S^{-1}S\cc M$ can be regarded as the set of objects of an additive
category in the usual way, and we use exactly the same notation to
denote that category. In this way $S^{-1}S\cc M$ becomes a
simplicial additive category. Its symmetric Grayson $K$-theory
spectrum will be denoted by $K^{Gr}(S^{-1}S\cc M)$. It follows from
the proof of~\cite[9.3]{Gr} that the map $\Ob\cc M\to S^{-1}S\cc M$
sending $M$ to $(M,0)$ induces a homotopy equivalence
   $$S^{\oplus}\cc M\to S^{\oplus}(S^{-1}S)\cc M.$$
Therefore the induced map of symmetric spectra
   $$K^{Gr}(\cc M)\to K^{Gr}(S^{-1}S\cc M)$$
is a stable weak equivalence, which is a level weak equivalence in
positive degrees.

Let $\cc A$ be a bivariant additive category. Then $S^{-1}S\cc A$ is
a simplicial bivariant additive category. For any $Y\in AffSm/k$ we
can form a bispectrum
   $$S^{-1}SA_Y=(S^{-1}SA_Y(0),S^{-1}SA_Y(1),\ldots),$$
where each $S^{-1}SA_Y(n)=K^{Gr}(C^\oplus S^{-1}S\cc A(-,Y\times\bb
G_m^{\wedge n}))$, and define a natural map of bispectra
   $$s:A_Y\to S^{-1}SA_Y.$$
This map is a level stable weak equivalence.

In order to construct $K$-theory spectra with entries being
sectionwise fibrant spaces, we use the category of topological
symmetric spectra $TopSp^\Sigma$ (see~\cite[section~I.1]{Sch}). We
can apply adjoint functors ``geometric realization", denoted by
$|-|$, and ``singular complex", denoted by $\cc S$, levelwise to go
back and forth between simplicial and topological symmetric spectra
   \begin{equation}\label{dwdw}
    |-|:Sp^\Sigma\rightleftarrows TopSp^\Sigma:\cc S.
   \end{equation}

\begin{rem}{\rm
By the standard abuse of notation $|-|$ denotes both the functor
from $Sp^\Sigma$ to $TopSp^\Sigma$ and the realization functor from
simplicial spectra to spectra. It will always be clear from the
context which of either meanings is used.

}\end{rem}

Given a (multisimplicial) additive category $\cc M$, denote by $\hat
K^{Gr}(\cc M)$ the symmetric spectrum $\cc S|K^{Gr}(\cc M)|$. The
unit of the adjunction induces a map of symmetric spectra
   $$K^{Gr}(\cc M)\to\hat K^{Gr}(\cc M),$$
functorial in $\cc M$. Observe that $|K^{Gr}(S^{-1}S\cc M)|$ is an
$\Omega$-spectrum in $TopSp^\Sigma$, and hence so is $\hat
K^{Gr}(S^{-1}S\cc M)$ in $Sp^\Sigma$.

Suppose $\cc A$ is a bivariant additive category satisfying the
property $(\ff{Aut})$. For any $Y\in AffSm/k$ we can form a
bispectrum
   $$S^{-1}S\hat A_Y=(S^{-1}S\hat A_Y(0),S^{-1}S\hat A_Y(1),\ldots),$$
where each $S^{-1}S\hat A_Y(n)=\hat K^{Gr}(C^\oplus S^{-1}S\cc
A(-,Y\times\bb G_m^{\wedge n}))$, and define a natural map of
bispectra
   $$t:S^{-1}SA_Y\to S^{-1}S\hat A_Y.$$
This map is a level stable weak equivalence.

By~\cite[9.4]{Gr} and~\cite[p.~16]{Wlk} there is a natural map of
symmetric topological spectra
   $$|K^{Gr}(C^\oplus S^{-1}S\cc A(X,Y\times\bb G_m^{\wedge 1}))|\to\Omega|K^{Gr}((S^{-1}S)(S^{-1}S)\cc A(X,Y))|.$$
It gives a natural map in $Pre^\Sigma(Sm/k)$
   $$v_1:S^{-1}S\hat A_Y(1)\to\Omega\hat K^{Gr}((S^{-1}S)(S^{-1}S)\cc A(-,Y)).$$
We can get more generally a map (see~\cite[p.~16]{Wlk} as
well)\footnotesize
   $$v_n:S^{-1}S\hat A_Y(n)\to\Omega\hat K^{Gr}(C^\oplus(S^{-1}S)(S^{-1}S)\cc A(-,Y\times\bb G^{\wedge n-1}_m))\to
     \cdots\to\Omega^n\hat K^{Gr}((S^{-1}S)^{n+1}\cc A(-,Y)).$$
\normalsize One sets
   $$\kappa_0:=\Omega_{\bb G}(v_1)\circ\hat a_0:\hat K^{Gr}(S^{-1}S\cc A(-,Y))\to\Omega_{\bb G}\Omega \hat K^{Gr}((S^{-1}S)^2\cc A(-,Y)),$$
where $\hat a_0:\hat A_Y(0)\to\Omega_{\bb G}\hat A_Y(1)$ is the
structure map. Applying the above construction to the
multisimplicial bivariant additive category $(S^{-1}S)^n\cc A$, we
get a map
   $$\kappa_n:\Omega^n\hat K^{Gr}((S^{-1}S)^{n+1}\cc A(-,Y))\to\Omega_{\bb G}\Omega^{n+1}\hat K^{Gr}((S^{-1}S)^{n+2}\cc A(-,Y)).$$

\begin{defs}{\rm
Let $\cc A$ be a bivariant additive category with
$(\ff{Aut})$-property and let $Y\in AffSm/k$. Then the bispectrum
$KGL_{\cc A,Y}$ is defined by the sequence in $Pre^\Sigma(Sm/k)$
   $$(\hat K^{Gr}((S^{-1}S)\cc A(-,Y)),\Omega\hat K^{Gr}((S^{-1}S)^2\cc A(-,Y)),
      \Omega^2\hat K^{Gr}((S^{-1}S)^3\cc A(-,Y)),\ldots).$$
Its structure maps are given by the maps $\kappa_n$-s.

}\end{defs}

The maps $v_n$-s determine a map of bispectra
   $$v:S^{-1}S\hat A_Y\to KGL_{\cc A,Y}.$$
So we have a map of bispectra
   \begin{equation}\label{tarara}
    \chi:=v\circ t\circ s:A_Y\to KGL_{\cc A,Y}.
   \end{equation}
In the next section we shall work with the bispectrum $KGL_{\cc
A,\spec k}$ for a certain bivariant additive category $\cc A$. It
will be shown that it represents Quillen's $K$-theory of algebraic
varieties.

\section{Comparing Grayson's and slice towers for $KGL$}

In this section we prove the main results of the paper. For
technical reasons we have dealt with general bivariant additive
categories so far. Below a concrete example of such a category $\cc
A$ is given. It will lead to solutions for the problems mentioned in
the introduction. Its definition is extracted from~\cite{GP1}. We
start with preparations.

Let $U,X\in Sm/k$. Define $\text{Supp}(U\times X/X)$ as the set of
all closed subsets in $U\times X$ of the form $A=\cup A_i$, where
each $A_i$ is a closed irreducible subset in $U\times X$ which is
finite and surjective over $U$. The empty subset $\empty$ in
$U\times X$ is also regarded as an element of $\text{Supp}(U\times
X/X)$.

Given $U,X \in Sm/k$ and $A\in\text{Supp}(U\times X/X)$, let
$I_A\subset\mathcal O_{U \times X}$ be the ideal sheaf of the closed
set $A \subset U \times X$. Denote by $A_m$ the closed subscheme in
$U \times X$ of the form $(A,\mathcal O_{U \times X}/I^m_A)$. If
$m=0$, then $A_m$ is the empty subscheme. Define
$\text{SubSch}(U\times X/X)$ as the set of all closed subschemes in
$U \times X$ of the form $A_m$.

For any $Z\in\text{SubSch}(U\times X/X)$ we write $p^Z_U:Z\to U$ to
denote $p\circ i$, where $i:Z\hookrightarrow U\times X$ is the
closed embedding and $p: U \times X \to U$ is the projection. If
there is no likelihood of confusion we shall write $p_U$ instead of
$p^Z_U$, dropping $Z$ from notation.

Clearly, for any $Z \in \text{SubSch}(U\times X/X)$ the reduced
scheme $Z^{red}$, regarded as a closed subset of $U\times X$,
belongs to $\text{Supp}(U\times X/X)$.

For any $U,X\in Sm/k$ we define objects of $\cc A(U,X)$ as
equivalence classes for the triples
   $$(n,Z,\phi:p_{U,*}(\cc O_Z)\to M_n(\cc O_U)),$$
where $n$ is a nonnegative integer, $Z \in \text{SubSch}(U\times
X/X)$ and $\phi$ is a non-unital homomorphism of sheaves of $\cc
O_U$-algebras. Let $p(\phi)$ be the idempotent $\phi(1)\in
M_n(\Gamma(U,\cc O_U))$, then $P(\phi)=\im(p(\phi))$ can be regarded
as a $p_{U,*}(\cc O_Z)$-module by means of $\phi$.

By definition, two triples $(n,Z,\phi)$, $(n',Z',\phi')$ are
equivalent if $n=n'$ and there is a triple $(n'',Z'',\phi'')$ such
that $n=n'=n''$, $Z,Z'\subset Z''$ are closed subschemes in $Z''$,
and the diagrams
   \begin{equation*}
    \xymatrix{p_{U,*}(\cc O_Z)\ar[rr]^\phi&&M_n(\cc O_U)&p_{U,*}'(\cc O_{Z'})\ar[rr]^{\phi'}&&M_n(\cc O_U)\\
              &p_{U,*}(\cc O_{Z''})\ar[ul]^{can}\ar[ur]_{\phi''}&&&p_{U,*}(\cc O_{Z''})\ar[ul]^{can}\ar[ur]_{\phi''}}
   \end{equation*}
are commutative. We shall often denote an equivalence class for the
triples by $\Phi$. Though $Z$ is not uniquely defined by $\Phi$,
nevertheless we shall also refer to $Z\subset U\times X$ as the {\it
support\/} of $\Phi$.

Given $\Phi,\Phi'\in\cc A(U,X)$ we first equalize supports $Z,Z'$ of
the objects $\Phi,\Phi'$ and then set
   $$\Hom_{\cc A(U,X)}(\Phi,\Phi')=\Hom_{p_{U,*}(\cc O_Z)}(P(\phi),P(\phi')).$$
Given any three objects $\Phi,\Phi', \Phi'' \in\cc A(U,X)$ a
composition law
   $$\Hom_{\cc A(U,X)}(\Phi,\Phi') \circ \Hom_{\cc A(U,X)}(\Phi',\Phi'') \to \Hom_{\cc A(U,X)}(\Phi,\Phi'')$$
is defined in the obvious way. This makes therefore $\cc A(U,X)$ an
additive category. The zero object is the equivalence class of the
triple $(0, \emptyset, 0)$. By definition,\footnotesize
   $$\Phi_1\oplus\Phi_2=(n_1+n_2,Z_1\cup Z_2,p_{U,*}(\cc O_{Z_1\cup Z_2})\to p_{U,*}(\cc O_{Z_1})\times p_{U,*}(\cc O_{Z_2})
     \to M_{n_1}(\cc O_U))\times M_{n_2}(\cc O_U))\hookrightarrow M_{n_1+n_2}(\cc O_U)).$$
\normalsize Clearly, $P(\phi_1\oplus\phi_2)\cong P(\phi_1)\oplus
P(\phi_2)$.

If $f: X^{\prime} \to X$ and $g: Y \to Y^{\prime}$ are in $Sm/k$,
then following~\cite[4.13]{GP1} and~\cite[4.14]{GP1} set
\begin{equation*}
\label{A(fg)}
\cc A(f,g)=f^* \circ g_*=g_*\circ f^*: \cc A(X,Y) \to \cc A(X^{\prime},Y^{\prime}).
\end{equation*}
By~\cite[4.14; 4.12]{GP1} the assignments $(X,Y)\mapsto \cc A(X,Y)$
and $(f,g)\mapsto \cc A(f,g)$ determine a functor
$$(U,X)\in (Sm/k)^{\op}\times Sm/k\mapsto\cc A(U,X)\in AddCats.$$
{\it Throughout this section by $\cc A$ we shall mean this bivariant
additive category}.

Next we shall introduce an action of $AffSm/k$ on $\cc A$ in the
sense of Section~\ref{Bivariant}. Following notation introduced just
below Theorem~4.15 of~\cite{GP1}, we set
\begin{equation*}
\alpha_U:= (1_U)^{\star}: \cc A(X,Y)\to\cc A(X\times U,Y\times U).
\end{equation*}
To check that the assignment $U \mapsto \alpha_U$ defines an action
of $AffSm/k$ on $\cc A$, we need to check commutativity of three
squares from Section~\ref{Bivariant}. Commutativity of the second
and the third squares
follow from~\cite[4.17; 4.18]{GP1}. Commutativity of the first
square, that is the equality $\cc A(1_{X\times U},1_Y\times f) \circ
\alpha_U= \cc A(1_{X\times U},g\times 1_U) \circ \alpha_V$, is
exactly commutativity of the diagram from~\cite[4.24]{GP1}. Thus the
assignment $U \mapsto \alpha_U$ defines an action of $AffSm/k$ on
$\cc A$. Below we shall consider the bivariant category $\cc A$
equipped with this specific action of $AffSm/k$.

For any $U,X \in Sm/k$ the bivariant category $\cc A$ produces a
simplicial additive category
   $$d\mapsto\cc A(U\times\Delta^d,X).$$
It is straightforward to check that this simplicial additive
category is a $k[\Delta]$-linear additive category in the sense
of~\cite[p.~158]{Gr}.

By~\cite[4.27; 4.28]{GP1} the bivariant category $\cc A$ also
satisfies the property ($\ff{Aut}$) from Section~\ref{GraysonTower}.

Now we have a spectral category $\bb K$ whose objects are those
of $Sm/k$, i.e. a category enriched over $Sp^\Sigma$. Its morphism
symmetric spectra are of the form
   $$\bb K(U,X)=K^{Gr}(\cc A(U,X)).$$
We refer the reader to~\cite{GP1} for details. One can associate a
ringoid $\bb K_0$ to it. By definition the objects of $\bb K_0$ are
those of $Sm/k$ and
   $$\bb K_0(U,X)=\pi_0(\bb K(U,X)),\quad U,X\in Sm/k.$$

In what follows we shall write $H_{\bb K}^{p,q}(U,\bb Z)$ to denote
$H_{\cc A}^{p,q}(U,\spec k)$. In Remark~\ref{isp} we defined
complexes of presheaves $C^*(A_{0,Y}(q))$ (indexing is
cohomological). We set,
   $$\bb Z^{\bb K}(q):=C^*(A_{0,\spec k}(q)[-q])_{\nis}.$$
It follows from Remark~\ref{isp} that
   $$H_{\bb K}^{p,q}(U,\bb Z)=H^p_{\nis}(U,\bb Z^{\bb K}(q)).$$
We shall also denote by $KGL$ the bispectrum $KGL_{\cc A,\spec k}$.

There is another ringoid which is important in our analysis. Let
$\tilde{\cc P}(U,X)$, $U,X\in Sm/k$, be the additive category of big
coherent $\cc O_{U\times X}$-modules $P$ such that $\supp P$ is
finite over $U$ and the coherent $\cc O_U$-module $(p_U)_*(P)$ is
locally free (see~\cite{FS,GP1,Gr}). We shall write $\tilde{\cc
P}(U)$ to denote $\tilde{\cc P}(U,\spec k)$. Define a ringoid
$K_0^\oplus$ as
   $$K_0^\oplus(U,X)=K_0(\tilde{\cc P}(U,X)),\quad U,X\in Sm/k.$$
Here the right hand side stands for the Grothendieck group of the
additive category $\tilde{\cc P}(U,X)$.

If $X$ is affine then there is a natural additive functor
(see~\cite{GP1} for details)
   $$F_{U,X}:\cc A(U,X)\to\tilde{\cc P}(U,X)$$
which is an equivalence of categories whenever $U$ is affine.
By~\cite{GP1} $F_{U,X}$ is functorial in $U$. These functors can
naturally be extended to a map of ringoids
   $$F:\bb K_0\to K_0^{\oplus}.$$

Given $n\geq 0$ we denote by $\bb Z^{Gr}(n)$ (respectively $\bb
Z(n)$) the Grayson complex~\cite{Sus,Wlk} corresponding to the
ringoid $K_0^\oplus$ (repectively the Suslin--Voevodsky~\cite{SV1}
complex corresponding to the ringoid $Cor$). The complexes are
defined in the same fashion as $\bb Z^{\bb K}(n)$. Recall that
motivic cohomology is defined as
   $$H^{p,q}_{\cc M}(X,\bb Z):=H^p_{\nis}(X,\bb Z(q)).$$

\begin{thm}[Suslin~\cite{Sus}]\label{suslin}
For any $n\geq 0$, the canonical homomorphism of complexes of
Nisnevich sheaves $\bb Z^{Gr}(n)\to\bb Z(n)$ is a quasi-isomorphism.
\end{thm}

\begin{cor}\label{nashe}
For any $n\geq 0$, the canonical homomorphism of complexes of
Nisnevich sheaves $\bb Z^{\bb K}(n)\to\bb Z^{Gr}(n)$, induced by the
map of ringoids $F:\bb K_0\to K_0^{\oplus}$, is an isomorphism.
Hence, for any smooth scheme $X\in Sm/k$, cohomology groups $H_{\bb
K}^{p,q}(X,\bb Z)$ coincide with motivic cohomology groups
$H^{p,q}_{\cc M}(X,\bb Z)$.
\end{cor}

\begin{proof}
As we have mentioned above, the additive functor $F_{U,X}:\cc
A(U,X)\to\cc P(U,X)$ is an equivalence whenever $U$ and $X$ are
affine. It follows that the map of simplicial abelian groups
   $$(d\mapsto\bb K_0(U\times\Delta^d,\bb G_m^{\times n}))\to(d\mapsto K_0^\oplus(U\times\Delta^d,\bb G_m^{\times n})),\quad n\geq 0,$$
is an isomorphism for every smooth affine scheme $U$. Hence the map
of simplicial abelian sheaves
   $$(d\mapsto\bb K_0(-\times\Delta^d,\bb G_m^{\times n})_{\nis})\to(d\mapsto K_0^\oplus(-\times\Delta^d,\bb G_m^{\times n})_{\nis})$$
is an isomorphism of motivic spaces. Our assertion now follows.
\end{proof}

\begin{cor}\label{abvgd}
The cohomology groups $H_{\bb K}^{*,*}(X,\bb Z)$ are homotopy
invariant and satisfy the cancelation property.
\end{cor}

\begin{proof}
This follows from~\cite[3.1; 4.13]{Sus} and Corollary~\ref{nashe}.
\end{proof}

\begin{cor}\label{efgh}
Let $k$ be a perfect field. Then $A_{0,\spec k}(n)=s_n(A_{0,\spec
k}(n))$ for each $n\geq 0$.
\end{cor}

\begin{proof}
This follows from Proposition~\ref{kahn}, Theorem~\ref{suslin} and
Corollary~\ref{nashe}.
\end{proof}

By definition, by the $K$-theory of $X$ we shall mean the Waldhausen
algebraic $K$-theory symmetric spectrum of big vector bundles (regarded as
an exact category)
   $$K(X)=K(\tilde{\cc P}(X)).$$
We set $G_X:=F_{X,\spec k}$. Observe that $G_X$ is functorial in
$X$. So we get a map in $Pre^\Sigma(Sm/k)$
   $$G:K^{Gr}(\cc A(-,\spec k))\to K(-),$$
where the left hand side spectrum is defined on p.~\pageref{KGr}.

\begin{prop}\label{kukuku}
$G$ is a Nisnevich local weak equivalence and induces canonical
isomorphisms
   $$K_p^{\cc A}(X,\spec k)\cong K_p(X),$$
for any smooth scheme $X$ and any integer $p$, where the left hand
side group is defined on p.~\pageref{KpA}.
\end{prop}

\begin{proof}
The fact that $G$ is a Nisnevich local weak equivalence follows from
the fact that $G_X$ is an equivalence of categories whenever $X$ is
affine. So we also have that
   $$K^{Gr}(\cc A(-\times\Delta^d,\spec k))\to K(-\times\Delta^d),\quad d\geq 0,$$
is a Nisnevich local weak equivalence in $Pre^\Sigma(Sm/k)$.

Consider a commutative diagram in $Pre^\Sigma(Sm/k)$
   \begin{equation}\label{hop}
   \xymatrix{
    K^{Gr}(\cc A(-,\spec k))\ar[r]\ar[d]_G&|K^{Gr}(\cc A(-\times\Delta^.,\spec k))|\ar[r]\ar[d]&|K^{Gr}(\cc A(-\times\Delta^.,\spec k))|_f\ar[d]^\gamma\\
    K(-)\ar[r]^\alpha&|K(-\times\Delta^.)|\ar[r]^\beta&|K(-\times\Delta^.)|_f}
   \end{equation}
Here the lower $f$-symbol refers to a fibrant replacement functor in
$Pre^\Sigma_{nis}(Sm/k)$. The vertical arrows are Nisnevich local
weak equivalences. The left horizontal arrows are motivic weak
equivalences by~\cite[3.8]{MV}.

Since $K(-)$ is homotopy invariant, then $\alpha$ is a stable weak
equivalence. By~\cite{T} $K(-)$ is Nisnevich excisive, and hence
$\beta,\gamma$ are stable weak equivalences. It remains to observe
that
   $$K_p^{\cc A}(X,\spec k)=\pi_p(|K^{Gr}(\cc A(-\times\Delta^.,\spec k))|_f(X))$$
for any $X\in Sm/k$.
\end{proof}

\begin{cor}\label{vladik}
Let $K(-)\to\tilde {K}(-)$ be any fibrant replacement of $K(-)$ in
the stable projective model structure of $Pre^\Sigma(Sm/k)$. Then
the composite map
   $$K^{Gr}(\cc A(-,\spec k))\bl G\to K(-)\to\tilde {K}(-)$$
is a motivic fibrant replacement of $K^{Gr}(\cc A(-,\spec k))$ in
$Pre^\Sigma_{mot}(Sm/k)$.
\end{cor}

\begin{proof}
All maps of diagram~\eqref{hop} are motivic weak equivalences. The
proof of the preceding proposition shows that $\tilde{K}(-)$ is
fibrant in $Pre^\Sigma_{mot}(Sm/k)$.
\end{proof}

We are now in a position to prove the following

\begin{thm}\label{grs1tower}
Let $k$ be a perfect field. Then the Grayson tower~\eqref{polez} of
$S^1$-spectra in $SH_{S^1}(k)$
    $$\cdots\to\Sigma_s^{q+1}A_{\spec k}^\Delta(q+1)\to\Sigma_s^qA^\Delta_{\spec k}(q)\to\cdots\to A^\Delta_{\spec k}$$
is isomorphic to the tower
    $$\cdots\to f_{q+1}(K(-))\to f_q(K(-))\to\cdots\to f_0(K(-)).$$
Moreover, $s_q(K(-))=EM(\bb Z(q))$ for every $q\geq 0$.
\end{thm}

\begin{proof}
This is a consequence of Theorems~\ref{grslice}, \ref{grslices},
\ref{suslin}, Corollaries~\ref{nashe}, \ref{abvgd}, \ref{efgh} and
Proposition~\ref{kukuku}.
\end{proof}

The next theorem says that the bispectrum $KGL$ represents algebraic
$K$-theory.

\begin{thm}\label{uhuhuh}
For any smooth scheme $X$ one has canonical isomorphisms
   $$KGL^{p,q}(X_+)=SH(k)(\Sigma^\infty_{\bb G}\Sigma_s^\infty X_+,\Sigma_s^{p-q}\Sigma_{\bb G}^qKGL)\cong K_{2q-p}(X),$$
where $K(X)$ is algebraic $K$-theory of $X$.
\end{thm}

\begin{proof}
Given a bispectrum $X$, let $X^\Delta$ be the bispectrum
$(|X_0(-\times\Delta^.)|,|X_1(-\times\Delta^.)|,\ldots)$. Taking a
fibrant replacement of $X^\Delta$ in the level Nisnevich local model
structure of $Pre^{\Sigma,\bb G}(Sm/k)$, we get a bispectrum
$X^\Delta_f$. So one has maps of bispectra
   $$X\to X^\Delta\to X^\Delta_f,$$
where the left arrow is a level motivic weak equivalence
by~\cite[3.8]{MV} and the right arrow is a level Nisnevich local
weak equivalence.

Consider the bispectra $(S^{-1}S\hat A_{\spec k}^\Delta)_f$ and
$KGL^\Delta_f$. Note that the first bispectrum is equivalent to
$\widetilde A_{\spec k}^\Delta$. We claim that each structure map
   $$\rho_n:(KGL^\Delta_f)_n=(KGL_n^\Delta)_f\to\Omega_{\bb G}(KGL_{n+1}^\Delta)_f$$
is a stable weak equivalence in $Pre^\Sigma(Sm/k)$.
Corollary~\ref{vladik} and~\cite[9.3]{Gr} imply each
$(KGL_n^\Delta)_f$ has homotopy type of $\Omega^n\tilde K(-)\in
Pre^\Sigma(Sm/k)$.

By construction, the map $\rho_0$ factors as
   $$S^{-1}S\hat A_{\spec k}(0)^\Delta_f\to\Omega_{\bb G}S^{-1}S\hat A_{\spec k}(1)^\Delta_f\to\Omega_{\bb G}(KGL_{1}^\Delta)_f.$$
Corollary~\ref{abvgd} and the Cancelation Theorem for
$K$-theory~\ref{novgorod} implies the left arrow is a stable weak
equivalence. It follows from~\cite[9.6]{Gr} that a homotopy cofiber
of the right arrow is $\Omega_{\bb G}\Omega(\wt A_{0,\spec
k}^\Delta)_0$. By~Corollary~\ref{nashe} we have
   $$\pi_{p-1}(\Omega_{\bb G}\Omega(\wt A_{0,\spec k}^\Delta)_0(X))\cong H^{p,0}_{\bb K}(X\wedge\bb G_m,\bb Z)\cong H^{p,0}_{\cc M}(X\wedge\bb G_m,\bb Z),\quad X\in Sm/k,p\in\bb Z.$$
The proof of~\cite[4.2]{VoeAppr} implies $H^{p,0}_{\cc M}(X\wedge\bb
G_m,\bb Z)=0$. So $\Omega_{\bb G}(\wt A_{0,\spec k}^\Delta)_0$ is
zero in $\Ho(Pre^\Sigma(Sm/k))$, and hence $\rho_0$ is a stable weak
equivalence. The fact that each $\rho_n$ is a stable weak
equivalence is proved in a similar fashion. The only difference with
$\rho_0$ is that one iterates the $S^{-1}S$-construction at each
step.

We conclude that $KGL^\Delta_f$ is a motivically fibrant bispectrum.
Therefore,
   \begin{gather*}
    KGL^{p,q}(X_+)=SH(k)(\Sigma^\infty_{\bb G}\Sigma_s^\infty X_+,\Sigma_s^{p-q}\Sigma_{\bb G}^qKGL^\Delta_f)
     \cong SH_{S^1}(k)(\Sigma^\infty_{\bb G}\Sigma_s^\infty X_+,\Sigma_s^{p-q}(KGL^\Delta_f)_q)\\
     \cong SH_{S^1}(k)(\Sigma^\infty_{\bb G}\Sigma_s^\infty X_+,\Sigma_s^{p-q}\Omega^q\tilde K(-))\cong K_{2q-p}(X),
   \end{gather*}
as was to be shown.
\end{proof}

\begin{lem}\label{birmingh}
Let $k$ be a perfect field. Then the bispectrum $A_{\spec k}$ is
isomorphic in $SH(k)$ to $f_0(KGL)$.
\end{lem}

\begin{proof}
It follows from Corollaries~\ref{abvgd}, \ref{efgh} and
Lemma~\ref{babah} that $A_{\spec k}$ is in $SH^{eff}(k)$. Then
map~\eqref{tarara} of bispectra $\chi:A_{\spec k}\to KGL$. factors
as
   $$A_{\spec k}\xrightarrow{\theta} f_0(KGL)\xrightarrow{\zeta} KGL.$$
For any $X\in Sm/k$ and any $p\in\bb Z$ the induced map
   $$\zeta_*:SH(k)(\Sigma^\infty_{\bb G}\Sigma_s^{\infty}X_+,\Sigma_s^{p}f_0(KGL))\to SH(k)(\Sigma^\infty_{\bb G}\Sigma_s^{\infty}X_+,\Sigma_s^{p}KGL)$$
is an isomorphism by construction of $f_0(KGL)$. On the other hand,
Theorem~\ref{uhuhuh} implies the induced map
   $$\chi_*:SH(k)(\Sigma^\infty_{\bb G}\Sigma_s^{\infty}X_+,\Sigma_s^{p}A_{\spec k})\to SH(k)(\Sigma^\infty_{\bb G}\Sigma_s^{\infty}X_+,\Sigma_s^{p}KGL)$$
is an isomorphism, and hence so is
   $$\theta_*:SH(k)(\Sigma^\infty_{\bb G}\Sigma_s^{\infty}X_+,\Sigma_s^{p}A_{\spec k})\to SH(k)(\Sigma^\infty_{\bb G}\Sigma_s^{\infty}X_+,\Sigma_s^{p}f_0(KGL)).$$
Since $\Sigma^\infty_{\bb G}\Sigma_s^{\infty}X_+$ generate the
compactly generated triangulated category $SH^{eff}(k)$, we conclude
that $\theta$ is an isomorphism in $SH(k)$.
\end{proof}

The following result gives an explicit model for the non-negative
part of the slice tower of the bispetrum $KGL$.

\begin{thm}\label{glavnaya}
Let $k$ be a perfect field. Then the tower~\eqref{neploho} of
bispectra in $SH(k)$
    $$\cdots\to\Sigma_s^{q+1}\Sigma^{q+1}_{\bb G}A_{\spec k}\to\Sigma_s^q\Sigma^q_{\bb G}A_{\spec k}\to\cdots\to A_{\spec k}$$
is isomorphic to the tower
    $$\cdots\to f_{q+1}(KGL)\to f_q(KGL)\to\cdots\to f_0(KGL).$$
\end{thm}

\begin{proof}
By Lemma~\ref{birmingh} there is an isomorphism $\theta:A_{\spec
k}\to f_0(KGL)$ in $SH(k)$. Suppose an isomorphism
$\theta_q:\Sigma_s^q\Sigma^q_{\bb G}A_{\spec k}\cong f_q(KGL)$,
$q\geq 0$, is constructed. Since $\Sigma_s^{q+1}\Sigma^{q+1}_{\bb
G}A_{\spec k}\in\Sigma^{q+1}_{\bb G}SH(k)$ and $s_q(KGL)$ is
orthogonal to $\Sigma^{q+1}_{\bb G}SH(k)$, it follows that there is
a unique morphism
   $$\theta_{q+1}:\Sigma_s^{q+1}\Sigma^{q+1}_{\bb G}A_{\spec k}\to f_{q+1}(KGL)$$
making the diagram
   $$\xymatrix{\Sigma_s^{q+1}\Sigma^{q+1}_{\bb G}A_{\spec k}\ar[r]\ar[d]_{\theta_{q+1}}&\Sigma_s^q\Sigma^q_{\bb G}A_{\spec k}\ar[d]^{\theta_q}\\
               f_{q+1}(KGL)\ar[r]&f_q(KGL)}$$
commutative. We claim that $\theta_{q+1}$ is an isomorphism in
$SH(k)$.

By Theorem~\ref{horosho} and Corollary~\ref{abvgd} a homotopy
cofiber of the upper horizontal arrow is $\Sigma_s^q\Sigma^q_{\bb
G}A_{0,\spec k}$. Therefore,
   $$SH(k)(\Sigma^{q+1}_{\bb G}\Sigma^\infty_{\bb G}\Sigma_s^\infty X_+,\Sigma_s^{p}\Sigma^q_{\bb G}A_{0,\spec k})=
     SH(k)(\Sigma_{\bb G}\Sigma^\infty_{\bb G}\Sigma_s^\infty X_+,\Sigma_s^{p}A_{0,\spec k})\cong H_{\bb K}^{p,0}(X\wedge\bb G_m,\bb Z)$$
for any $X\in Sm/k$ and integer $p$. The proof of
Theorem~\ref{uhuhuh} shows that $H_{\bb K}^{p,0}(X\wedge\bb G_m,\bb
Z)=0$, and hence $f_{q+1}(\Sigma_s^{q}\Sigma^q_{\bb G}A_{0,\spec
k}))=0$.

Since $f_{q+1}(s_q(KGL))=0$, we see that the horizontal arrows of
the commutative diagram
      $$\xymatrix{f_{q+1}(\Sigma_s^{q+1}\Sigma^{q+1}_{\bb G}A_{\spec k})\ar[r]\ar[d]_{f_{q+1}(\theta_{q+1})}
               &f_{q+1}(\Sigma_s^q\Sigma^q_{\bb G}A_{\spec k})\ar[d]^{f_{q+1}(\theta_q)}\\
               f_{q+1}(f_{q+1}(KGL))\ar[r]&f_{q+1}(f_q(KGL))}$$
are isomorphisms. But $f_{q+1}(\theta_q)$ is an isomorphism, and
hence so is $f_{q+1}(\theta_{q+1})$. Lemma~\ref{birmingh} implies
$\Sigma_s^{q+1}\Sigma^{q+1}_{\bb G}A_{\spec k}$ is in
$\Sigma^{q+1}_{\bb G}SH(k)$. Since $f_{q+1}(KGL)$ belongs to
$\Sigma^{q+1}_{\bb G}SH(k)$ as well and $f_{q+1}(\theta_{q+1})$ is
an isomorphism, we conclude that $\theta_{q+1}$ is an isomorphism.
\end{proof}

One of the equivalent models for the motivic Eilenberg--Mac~Lane
bispectrum $H_{\bb Z}$ is as follows. Let $Cor$ be the ringoid of
finite correspondences over $Sm/k$ (see, e.g.,~\cite{SV1}). The cube
of sheaves $Cor(-,\bb G_m^{\wedge n})$ is defined similar to the
cube $K_0^{Gr}(\cc A(-,Y\times\bb G_m^{\wedge n}))$. Its vertices
are sheaves $Cor(-,\bb G_m^{\times k})$, $k\leq n$. By definition,
   $$H_{\bb Z}=(EM(Cor(-,\spec k)),EM(C^\oplus Cor(-,\bb G_m^{\wedge 1})),\ldots),$$
where $EM$ stands for the Eilenberg--Mac~Lane functor in the sense
of~\cite[Appendix~A]{GP} from abelian groups to $Sp^\Sigma$.

The composite map of ringoids
   $$\bb K_0\xrightarrow{F}K_0^\oplus\to Cor$$
(see~\cite{Sus,Wlk} for the definition of the second arrow) yields a
map of bispectra
   $$\lambda:A_{0,\spec k}\to H_{\bb Z}.$$
The proof of Theorem~\ref{horosho} and
Corollaries~\ref{nashe}-\ref{abvgd} shows that $\lambda$ is an
isomorphism in $SH(k)$.

The next result was first conjectured by
Voevodsky~\cite{VoeProbl,VoeAppr} and solved by Levine~\cite{Levcon}
by using the coniveau tower (over perfect fields).

\begin{thm}\label{voevslices}
Let $k$ be a perfect field. Then for every $q\geq 0$ we have
isomorphisms in $SH(k)$
   $$s_q(KGL)\cong\Sigma_s^q\Sigma_{\bb G}^q H_{\bb Z}.$$
\end{thm}

\begin{proof}
The proof of Theorem~\ref{glavnaya} shows that there is a
commutative diagram in $SH(k)$
   $$\xymatrix{\Sigma_s^{q+1}\Sigma^{q+1}_{\bb G}A_{\spec k}\ar[r]\ar[d]_{\theta_{q+1}}&\Sigma_s^q\Sigma^q_{\bb G}A_{\spec k}\ar[d]^{\theta_q}
               \ar[r]&\Sigma_s^q\Sigma^q_{\bb G}A_{0,\spec k}\ar[r]&\Sigma_s^{q+2}\Sigma^{q+1}_{\bb G}A_{\spec k}\ar[d]^{\Sigma_s\theta_{q+1}}\\
               f_{q+1}(KGL)\ar[r]&f_q(KGL)\ar[r]&s_q(KGL)\ar[r]&\Sigma_sf_{q+1}(KGL),}$$
where the vertical arrows are isomorphisms. Since $SH(k)$ is
triangulated, then there exists an isomorphism
   $$\Sigma_s^q\Sigma^q_{\bb G}A_{0,\spec k}\cong s_q(KGL).$$
It remains to observe that $\lambda:A_{0,\spec k}\to H_{\bb Z}$
induces an isomorphism $\Sigma_s^q\Sigma^q_{\bb G}A_{0,\spec
k}\cong\Sigma_s^q\Sigma_{\bb G}^q H_{\bb Z}$ in $SH(k)$.
\end{proof}

Let $\tilde{\cc P}(\bb G_m^{\times q})(X)$ be the additive category
whose objects are the tuples $(P,\theta_1,\ldots,\theta_q)$ with
$P\in\tilde{\cc P}(X)$ and $(\theta_1,\ldots,\theta_q)$ commuting
automorphisms. The cube of affine schemes $\bb G_m^{\wedge q}$ gives
rise to a cube of additive categories $\tilde{\cc P}(\bb G_m^{\wedge
q})(X)$ with vertices being $\tilde{\cc P}(\bb G_m^{\times q})(X)$,
$0\leq k\leq q$. The edges of the cube are given by the additive
functors $i_s:\tilde{\cc P}(\bb G_m^{\times k-1})(X)\to\tilde{\cc
P}(\bb G_m^{\times k})(X)$
   $$(P,(\theta_1,\ldots,\theta_{k-1}))\longmapsto(P,(\theta_1,\ldots,1,\ldots,\theta_{k-1})),$$
where 1 is the $s$th coordinate.

Grayson's machinery~\cite{Gr} (see~\cite{Wlk,GP} as well) produces a
tower in $\Ho(Pre^\Sigma_{nis}(Sm/k))$
   \begin{equation}\label{grdan}
    \cdots\to\Sigma^{q}_s|K^{Gr}(C^\oplus\tilde{\cc P}(\bb G_m^{\wedge q})(-\times\Delta^.))|\to\cdots\to |K^{Gr}(\tilde{\cc P}(-\times\Delta^.))|.
   \end{equation}
By~\cite[10.5]{Gr} $|K^{Gr}(\tilde{\cc
P}(-\times\Delta^.))|=|K(\tilde{\cc P}(-\times\Delta^.))|$. This
tower produces the {\it Grayson motivic spectral sequence\/} for
$\tilde{\cc P}(X)$ (see~\cite{Gr,Sus,Wlk,GP})
   \begin{equation}\label{graysonmss}
    E_2^{pq}=H^{p-q}_{\nis}(X,\bb Z^{Gr}(-q))\Longrightarrow K_{-p-q}(X),\quad X\in Sm/k.
   \end{equation}
In view of Theorem~\ref{suslin} it takes the form
   $$E_2^{pq}=H^{p-q,-q}_{\cc M}(X,\bb Z)\Longrightarrow K_{-p-q}(X),\quad X\in Sm/k.$$

We are now in a position to prove the following

\begin{thm}\label{ochenhorosho}
Let $k$ be a perfect field. Then the Grayson motivic spectral
sequence~\eqref{graysonmss} is isomorphic to the Voevodsky motivic
spectral sequence~\cite[p.~171]{DLORV}
   $$E_2^{pq}=SH(k)(\Sigma^\infty_{\bb G}\Sigma_s^{\infty}X_+,\Sigma_s^{p-q}\Sigma^q_{\bb G}s_0(KGL))\Longrightarrow K_{-p-q}(X),$$
produced by the slice tower for the bispectrum $KGL$.
\end{thm}

\begin{proof}
Recall that there is an additive functor $G_X:\cc A(X,\spec
k)\to\tilde{\cc P}(X)$, functorial in $X$, which is an equivalence
whenever $X$ is affine. It induces a map of multisimplicial additive
categories
   $$G_{q,X}:C^\oplus\cc A(X,\spec k)(\bb G_m^{\wedge q})\to C^\oplus\tilde{\cc P}(\bb G_m^{\wedge q})(X),$$
which is an equivalence whenever $X$ is affine. In view of
$(\ff{Aut})$-property for $\cc A$, we can identify $\cc A(X,\spec
k)(\bb G_m^{\wedge q})$ with $\cc A(X,\spec k\times\bb G_m^{\wedge
q})$.

It follows that Grayson's tower~\eqref{grdan} for $\tilde{\cc P}(X)$
is isomorphic in $\Ho(Pre^\Sigma_{nis}(Sm/k))$ to Grayson's
tower~\eqref{polez} for $\cc A(X,\spec k)$. Corollary~\ref{nashe}
and Proposition~\ref{kukuku} imply that Grayson's motivic spectral
sequence~\eqref{graysonmss} is isomorphic to Grayson's motivic
spectral sequence~\eqref{grspseq} for $\cc A$
   $$E_2^{pq}=H^{p-q,-q}_{\bb K}(X,\bb Z)\Longrightarrow K_{-p-q}^{\cc A}(X,\spec k),\quad X\in Sm/k.$$
Theorems~\ref{horosho}, \ref{glavnaya}, \ref{voevslices} now finish
the proof.
\end{proof}

\appendix\section{Some facts on spectra}

We prove here a couple of useful facts. First we wish to compare the
agreement of the bispectrum $KGL$ with the classical $K$-theory $\bb
P^1$-spectrum $BGL$ (see, e.g., \cite{VoeICM,MV,PPO}).

The functor $diag:SH_{S^1,\bb G}(k)\to SH_{S^1\wedge\bb G}(k)$
sending a bispectrum to its diagonal $S^1\wedge\bb G$-spectrum is an
equivalence of categories. In particular, $diag(KGL)$ is isomorphic
to the following $S^1\wedge\bb G$-spectrum:\footnotesize
   $$KGL_1=(\wh K^{Gr}(S^{-1}S\cc A(-,\spec k))_f^{(0)},\Omega\wh K^{Gr}((S^{-1}S)^2\cc A(-,\spec k))_f^{(1)},\Omega^2\wh K^{Gr}((S^{-1}S)^3\cc A(-,\spec k))_f^{(2)},\ldots),$$
\normalsize where $f$ refers to motivic fibrant replacement with
respect to the injective model structure of motivic spaces
(see~\cite{Jar2}) and the superscript ${}^{(n)}$ refers to the $n$th
space of the $S^1$-spectrum $\Omega^n\wh K^{Gr}((S^{-1}S)^{n+1}\cc
A(-,\spec k))$. Let $\cc K$ be a motivic fibrant replacement of the
$K$-theory presheaf $U\mapsto K(\tilde{\cc P}(U))$. Then we can
choose homotopy equivalences
   $$\alpha_n:\Omega^n\wh K^{Gr}((S^{-1}S)^{n+1}\cc A(-,\spec k))_f^{(n)}\to\cc K,\quad n\geq 0,$$
and maps $\beta_n:\cc K\wedge(S^1\wedge\bb G)\to\cc K$ which are
defined as
   \begin{gather*}
    \cc K\wedge(S^1\wedge\bb G)\xrightarrow{\gamma_n\wedge 1}\Omega^n\wh K^{Gr}((S^{-1}S)^{n+1}\cc A(-,\spec k))_f^{(n)}\wedge(S^1\wedge\bb G)\to\\
    \Omega^{n+1}\wh K^{Gr}((S^{-1}S)^{n+2}\cc A(-,\spec k))_f^{(n+1)}\xrightarrow{\alpha_{n+1}}\cc K
   \end{gather*}
with $\gamma_n$ homotopy inverse of $\alpha_n$. We then get a
$S^1\wedge\bb G$-spectrum
   $$KGL_2=(\cc K,\cc K,\cc K,\ldots)$$
with structure maps given by $\beta_n$-s.

It follows from~\cite[6.3]{Riou} that $KGL_1$ and $KGL_2$ are
isomorphic in $SH_{S^1\wedge\bb G}(k)$. By the same result
and~\cite[1.1.2]{Riou1} $KGL_2$ is isomorphic in $SH_{S^1\wedge\bb
G}(k)$ to the spectrum
   $$KGL_3=(\cc K,\cc K,\cc K,\ldots)$$
with each structure map given by $\beta_0$.

There is a zigzag of motivic equivalences
   $$S^1\wedge\bb G\xleftarrow{\sim}_{\bb A^1}\wt T\xrightarrow{\sim}_{\bb A^1}T\xleftarrow{\sim}_{\bb A^1}\bb P^{ 1},$$
where $\wt T$ is the mapping cylinder for the inclusion $(\bb
G_m)_+\to\bb A^1_+$. By~\cite[2.13]{Jar2} the zigzag induces an
equivalence of categories
   $$\theta:SH_{S^1\wedge\bb G}(k)\to SH_{\bb P^1}(k).$$
Consider a $\bb P^1$-spectrum
   $$KGL_4=(\cc K,\cc K,\cc K,\ldots),$$
where each structure map $\cc K\to\Omega_{\bb P^1}\cc K$ is given by
   $$\cc K\to\Omega_{S^1\wedge\bb G}\cc K\lra{\sim}\Omega_{\bb P^1}\cc K.$$
Here the left arrow is adjoint to $\beta_0$ and the right arrow is a
chosen homotopy equivalence induced by the zigzag above (recall that
$\cc K$ is a motivically fibrant space).

It follows from~\cite[6.3]{Riou} that $\theta(KGL_3)$ is isomorphic
to $KGL_4$ in $SH_{\bb P^1}(k)$. It remains to
apply~\cite[1.1.2]{Riou1} to show that $KGL_4$ is isomorphic in
$SH_{\bb P^1}(k)$ to the $\bb P^1$-spectrum $BGL$ defined
in~\cite[1.2.1]{PPO}.

We document the above arguments as follows.

\begin{thm}\label{kglagree}
The image of the bispectrum $KGL$ under the equivalence of
triangulated categories $\theta\circ diag:SH_{S^1,\bb G}(k)\to
SH_{\bb P^1}(k)$ is isomorphic to the $K$-theory $\bb P^1$-spectrum
$BGL$ in the sense of~\cite[1.2.1]{PPO}.
\end{thm}

Although the authors have not found the following result in the
literature, they do not have pretensions to originality. It is used
in the proof of Lemma~\ref{nesv}.

\begin{prop}\label{gmloops}
If $E$ is a connected motivically fibrant $S^1$-spectrum, then so is
$\Omega_{\bb G}E$.
\end{prop}

\begin{proof}
Clearly, $\Omega_{\bb G}E$ is motivically fibrant. To prove that it
is connected, it suffices to check that for any smooth local
Henselian scheme $U$ one has $\pi_{n<0}(E(\bb G_{m,U}))=0$ (recall
that $\bb G$ is sectionwise equivalent to the pointed motivic space
$(\bb G_m,1)$). Since $\Omega_{\bb G_m}E$ is motivically fibrant,
by~\cite[6.1.6]{Mor} it is enough to verify that for any $k$-smooth
variety $X$ one has $\pi_{n<0}(E(\bb G_{m,K}))=0$ with $K=k(X)$ its
function field.

\begin{sublem}
If $\cc F$ is a strictly homotopy invariant Nisnevich sheaf of
Abelian groups on $Sm/k$, then $H^n_{\nis}(\bb G_{m,k(X)},\cc F)=0$
for all $n>0$ and $X\in Sm/k$.
\end{sublem}

\begin{proof}
The result is well-known for $n>1$. One has,
   \begin{gather*}
    H^1_{\nis}(\bb G_{m,k(X)},\cc F)\bl{(1)}=H^2_{\nis}(S^1\wedge\bb G_{m,k(X)},\cc F)\bl{(2)}=[S^1\wedge\bb G_{m,k(X)},K(\cc
    F,2)]_{H_{\bb A^1}(k)}\bl{(3)}=\\ =[\bb P^1_{k(X)},K(\cc
    F,2)]_{H_{\bb A^1}(k)}\bl{(4)}=H^2_{\nis}(\bb P^1,\cc
    F)\bl{(5)}=0.
   \end{gather*}
Here $(1)$ is given by the suspension isomorphism, $(2)$ holds
because $K(\cc F,2)$ is an $\bb A^1$-local motivic space,  $(3)$
holds because $S^1\wedge\bb G_{m,k(X)}\cong\bb P^1_{k(X)}$ in
$H_{\bb A^1}(k)$, $(4)$ holds because $K(\cc F,2)$ is an $\bb
A^1$-local motivic space, and finally $(5)$ follows from the fact
that $\dim\bb P^1=1<2$.
\end{proof}

Now the spectral sequence
   $$H^p(\bb G_{m,k(X)},\underline{\pi}_q(E))\Longrightarrow\pi_{q-p}(E(\bb G_{m,k(X)}))$$
together with the sublemma above show that $H^0(\bb
G_{m,k(X)},\underline{\pi}_q(E))=\pi_q(E(\bb G_{m,k(X)}))=0$ for
$q<0$, because $\underline{\pi}_{q<0}(E)=0$.
\end{proof}

To conclude the paper, we remark that all presheaves of symmetric
spectra forming the main bispectra $A_{\spec k},A_{0,\spec k}$ we
work with are $\bb K$-modules in the sense of~\cite{GP1}. Moreover,
their structure maps are $\bb K$-module morphisms. Also, Grayson's
tower~\eqref{polez} for $\cc A$ is in fact a tower in the homotopy
category $\Ho(\Mod\bb K)$ of $\bb K$-modules. It produces a tower of
compact objects in the motivic homotopy category of $\bb K$-modules
in the sense of~\cite{GP}. This point of view to the motivic
spectral sequence motivated the authors to develop the ``enriched
motivic homotopy theory" of spectral categories and modules over
them~\cite{GP,GP1}. As an application, the motivic spectral sequence
is realized in associated triangulated categories. Though we tried
to avoid the use of this language here, it is this theory that led
the authors to the main results of this paper.

\end{document}